\documentclass[a4paper,11pt]{article}

\usepackage{graphicx}
\usepackage{enumerate}
\usepackage[margin=2.4cm]{geometry}
\usepackage{amsthm, amssymb}
\usepackage{mathrsfs} 
\usepackage{soul}

\usepackage{mathtools}

\usepackage{amsmath}

\usepackage{dsfont}
\usepackage{amsfonts}
\usepackage[normalem]{ulem}
\usepackage{amscd} % Rectangular diagrams

\usepackage{subcaption}

\usepackage[backref]{hyperref}

\newcommand{\eqdef}{\stackrel{\mathrm{def}}{=}}
\newcommand{\EE}{\mathbb{E}}
\newcommand{\PP}{\mathbb{P}}

\newcommand{\RR}{\mathbb{R}}

\newcommand{\NN}{\mathbb{N}}
\newcommand{\ZZ}{\mathbb{Z}}

\usepackage{todonotes}

\newcommand\norm[1]{\left\lVert#1\right\rVert}

\newtheorem{theorem}{Theorem}
\newtheorem{corollary}{Corollary}
\newtheorem{proposition}{Proposition}
\newtheorem{lemma}{Lemma}
\newtheorem{remark}{Remark}

\theoremstyle{definition}

\title{Collisions of random walks in unimodular random graphs: applications to the random geometric graph and long-range percolation}
\author{Jhon Astoquillca \thanks{Institute of Mathematics, Statistics and Computer Science, University of S\~ao Paulo, Brazil} \and Carlos Martinez-Arevalo \thanks{Institute for Applied Mathematics, University of Bonn, Endenicher Allee 60, 53115 Bonn, Germany}}

\begin{document}

\maketitle

\begin{abstract}
We study collision properties of simple random walks in unimodular random rooted graphs. This work continues the study initiated in~\cite{HutchcroftPeres2015}: under recurrence and an integrability condition on the root, two independent random walks collide infinitely often a.s. We prove that, under transience and an integrability assumption involving the Green function on the root, two independent random walks collide only finitely often a.s. 

We apply these results to several random graphs with unbounded degree: the Gilbert graph, the Delaunay graph, the Gabriel graph; and the long-range percolation model. We use these collision properties to characterize stationary measures of the voter model on these graphs.
\end{abstract}

\section{Introduction}
Collision properties of random walks serve as an important tool in the study of systems of random walks, including coalescing, annihilating, and branching random walks, which arise naturally in the study of interacting particle systems. We say that a locally finite graph~$G=(V,E)$ has the \emph{finite (resp. infinite) collision property} if two independent discrete-time random walks with transition function
\begin{equation}\label{eq_p_transition_srw_intr}
p(x,y) = \deg(x)^{-1} \cdot \mathds{1}\{ \{x,y\} \in E \} , \quad x,y \in V,    
\end{equation}
started from the same vertex collide finitely (resp.~infinitely) many times almost surely, where a collision occurs whenever the two walks occupy the same vertex at the same time. The corresponding definition for continuous-time random walks with jump rate~\eqref{eq_p_transition_srw_intr} is analogous; we then speak of the discrete or continuous collision property. Precise definitions and relations between these notions are given in Section~\ref{ss_collision_properties}. 

Collision properties are closely connected to recurrence and transience of the underlying graph. In fact, on vertex-transitive graphs, the infinite collision property is equivalent to recurrence, while the finite collision property is equivalent to transience. However, on more general graphs, these equivalences break down: certain implications continue to hold, whereas others fail completely. For bounded-degree graphs, the infinite collision property implies recurrence, while transience implies the finite collision property. The converses, however, do not hold in general. In particular, the Comb lattice is recurrent and possesses the finite collision property as proved in~\cite{PeresKrishnapur04}. Moreover, the same article provides, in the Introduction, an example of a transient graph with the infinite collision property once the bounded-degree assumption is removed, showing that, in general, no implication holds in either direction. 

Rather than attempting a comprehensive overview of the existing literature, we refer the reader to~\cite[Related work]{Umb25}, which contains an extensive account of previous results on collision properties of random walks. In particular, we mention~\cite{Chen08}, where the authors investigate when the discrete and continuous collision properties are equivalent. They prove this equivalence for quasi-transitive graphs with subexponential growth. We discuss this question further in Section~\ref{ss_collision_properties}.

Motivated by these observations, it is natural to investigate whether some of the relationships between collision properties and recurrence/transience persist in graph classes beyond the bounded-degree or vertex-transitive setting. To this end, we consider unimodular random rooted graphs, that is, random rooted graphs satisfying the mass-transport principle, a framework that extends several symmetry properties on average. A precise definition of unimodularity is given in Section~\ref{ss_unimodularity}.

This question was previously studied in~\cite{HutchcroftPeres2015}, where the authors considered continuous-time random walks with symmetric jump rates along edges. Using the same method, we obtain the analogous result for continuous-time random walks with jump rates given by~\eqref{eq_p_transition_srw}.

\begin{theorem}{\cite{HutchcroftPeres2015}}\label{thm_unimodularity_ICP}
Let~$(G,\rho)$ be a recurrent unimodular random rooted graph with~$\mathbf E[ \deg(\rho) ] <  \infty$. Then~$G$ has the discrete and continuous infinite collision property almost surely. 
\end{theorem}

We now present the central result of the paper: we obtain analogous results for the finite collision property, under a stronger assumption than in the theorem above, which will be enough for our applications. We write by~$p_n$ the~$n$-step transition function of~$p$ from~\eqref{eq_p_transition_srw_intr}.

\begin{theorem}\label{thm_unimodularity_FCP}
Let~$(G,\rho)$ be a unimodular random rooted graph with~$\mathbf E[ \deg(\rho) \cdot \sum_{n \ge 0} p_n(\rho,\rho) ] <  \infty$. Then~$G$ has the discrete and continuous finite collision property almost surely. 
\end{theorem}
From the hypothesis it follows that~$\mathbf{E}[\deg(\rho)]<\infty$, as well as~$\mathbf{E}[\sum_{n \ge 0} p_n(\rho,\rho) ]<\infty$, which in turn implies that~$G$ is transient almost surely. These properties align precisely with the assumptions of Theorem~\ref{thm_unimodularity_ICP}. However, it remains unclear whether this assumption can be weakened to merely require that~$\mathbf{E}[\deg(\rho)]<\infty$ and almost surely transience. 

Although the proof of Theorem~\ref{thm_unimodularity_FCP} is relatively simple, it is central to our approach, as it enables the study of the finite collision property for several classes of random graphs with unbounded degree a.s. 

Given a graph~$G$, the \emph{tail}~$\sigma$-algebra of a discrete-time simple random walk~$(X_n)_{n \ge 0}$ on~$G$ is defined by
$$\mathcal T_\mathrm{rw}(G) := \bigcap_{n \ge 0} \sigma( X_n,X_{n+1},X_{n+2},\dots ).$$
When~$\mathcal T_\mathrm{rw}(G)$ is trivial, the collision properties become~0--1 events. Although this assumption is not needed for the proofs of our main results, it will be used to derive an application to the voter model on these graphs; see Corollary~\ref{cor:vm}.

The focus of this article will be on the following two classes of random graphs. \\[-1.5mm]

\noindent \textbf{Graphs generated by a Poisson point process on~$\RR^d$.} Let~$\mathcal P$ be a homogeneous Poisson point process with intensity~$1$ on~$\RR^d$,~$d \ge 2$. Given~$\mathcal P$, we consider the following graphs: the infinite connected component of the random geometric graph; denoted by~$\mathcal G_\mathrm{SGil}$; the Delaunay triangulation graph~$\mathcal G_\mathrm{Del}$; and the Gabriel graph~$\mathcal G_\mathrm{Gab}$. See the definitions in Section~\ref{ss_rgg_et_all}. 

\begin{theorem}\label{thm_main_dim}
Let~$G \in \{ \mathcal G_\mathrm{SGil},\mathcal G_\mathrm{Del}, \mathcal G_\mathrm{Gab} \}$. Then,
\begin{enumerate}
    \item for~$d=2$, the graph~$G$ has the discrete and continuous infinite collision property almost surely;
    \item for~$d \ge 3$, the graph~$G$ has the discrete and continuous finite collision property almost surely;
    \item for~$d \ge 2$, the tail~$\sigma$-algebra~$\mathcal T_\mathrm{rw}(G)$ is trivial almost surely. 
\end{enumerate}
\end{theorem}
The previous theorem is reminiscent of the behavior of the infinite cluster of nearest-neighbor bond percolation on~$\mathbb Z^d$. In dimension~$d=2$, the infinite cluster has the infinite collision property~\cite{BarlowPeresSousi2012,ChenChen10}, whereas for~$d \ge 3$ it has the finite collision property. The latter follows from the fact that the graph is bounded degree and transient almost surely~\cite{Grimmett93}. \\

\noindent \textbf{Long-range percolation.} Consider the graph obtained from~$\ZZ^d$ by adding an edge between each pair of distinct vertices~$x,y \in \ZZ^d$ with probability~$p_{x,y}$ independently of the other pairs. Assume that~$p_{x,y} = \mathsf P_{\beta,s}(|x-y|)$ for some function~$\mathsf P_{\beta,s}$ with parameters~$\beta > 0$ and~$s > 0$ that satisfies
$$ \lim_{r \to \infty} \frac{\mathsf P_{\beta,s}(r) }{\beta r^{-s}} = 1.$$
In the literature, it is known that for~$d=1$ with~$s\in (1,2]$ and for~$d \ge 2$ with~$s >d$, there exists~$\beta > 0$ sufficiently large such that one can choose~$\mathsf P_{\beta,s}$ so that the percolation model admits a unique infinite connected component almost surely. See Section~\ref{ss_long_range_perc} for more details on the percolation regime and references.

\begin{theorem}\label{thm:long_range_cluster}
For each regime, we assume that the percolation model with edge probabilities~$\mathsf P_{\beta,s}$ admits a unique infinite connected component~$\mathcal C_\infty$ almost surely.
\begin{enumerate}
    \item For~$d = 1$ with~$s =2$, and~$d=2$ with~$s \ge 4$, the long-range percolation cluster~$\mathcal C_\infty$ has the discrete and continuous infinite collision property almost surely.
    \item For~$d = 1$ with~$s \in (1,2)$, and~$d \ge 2$ with~$s \in (d,d+2)$, the long-range percolation cluster~$\mathcal C_\infty$ has the discrete and continuous finite collision property almost surely.
    \item \cite{BenjaminiCurien2012} For~$d \ge 1$ with~$s \in (d,2d)$, the tail~$\sigma$-algebra~$\mathcal T_\mathrm{rw}(\mathcal C_\infty)$ is trivial almost surely.
\end{enumerate}
\end{theorem}
For the second part of the theorem, we restrict to~$d \ge 1$ and~$s \in (d,(d+2)\wedge 2d)$, as our argument relies on heat kernel estimates together with quantitative control on the tails of the random time after which these estimates hold, as established in~\cite{Sly12}. We remark that heat kernel estimates are also available for the regimes~$d=1$ with~$s>2$ and~$d\ge2$ with~$s\ge d+2$, see~\cite{CanCroydonKumagai2022}. However, in these cases, comparable control on the tails of this random time is not available.

As an application of the previous theorems, we obtain a complete characterization of the stationary distributions of the voter model on those random graphs. \\

\noindent \textbf{The voter model.} The voter model on a locally finite graph~$G$, denoted by~$(\eta_t)_{t \ge 0}$, is a Markov process with configuration space~$\{0,1\}^V$ and stochastic dynamics described informally as follows: each vertex (voter)~$x \in V$ updates its current state (opinion)~$\eta_t(x) \in \{0,1\}$ at rate~1 by uniformly choosing a neighbouring vertex~$y$ and adopting its state~$\eta_t(y)$. For a formal definition and comprehensive introduction, we refer the reader to~\cite{Liggett85,Swart2026}.  

In the voter model literature, it is well known that the evolution of the opinions of a group of voters is dual to a system of coalescing random walks. This duality implies the following. For~$\alpha \in [0,1]$, if the initial configuration~$\eta_0$ is distributed according to the Bernoulli($\alpha$) product measure on~$\{0,1\}^V$, the law of~$\eta_t$ converges weakly, as~$t \to \infty$, to a limiting probability measure~$\mu_\alpha$ which in turn is stationary for the voter model dynamics. We then have a one-parameter family of distinct stationary distributions~$\{\mu_\alpha: \alpha \in [0,1]\}$. The measures~$\mu_0$ and~$\mu_1$ correspond to the Dirac masses on the all-0 and all-1 configurations, respectively, which we refer to as the consensus measures. 

A stationary distribution is called extremal if it cannot be expressed as a nontrivial convex combination of other stationary distributions. We denote by~$\mathscr I_e(G)$ the set of extremal stationary distributions. It is known~\cite[Proposition 1.8.c]{Liggett85} that the set of stationary distributions of the voter model coincides with the closed convex hull of~$\mathscr I_e(G)$. Consequently, the structure of the stationary distributions of the voter model is completely characterized once the set~$\mathscr I_e(G)$ is described. The structure of~$\mathscr I_e(G)$ depends strongly on the collision properties of~$G$, as reflected in the following result. 

This statement follows from a stronger theorem proved in~\cite[Theorem 1.8. Chapter~V]{Liggett85}, and we present it in the form given in~\cite[Theorem 2.2]{Astoquillca26}, where a direct proof is also available.

\begin{theorem}\label{thm_voter_model_stat_meas}
Let~$G$ be a locally finite and connected graph.
\begin{enumerate}
    \item If~$G$ has the continuous infinite collision property, then~$\mathscr I_e(G) = \{ \mu_0,\mu_1 \}$. 
    \item If~$G$ has the continuous finite collision property and~$\mathcal T_\mathrm{rw}(G)$ is trivial, then~$\mathscr I_e(G) = \{ \mu_\alpha: \alpha \in [0,1] \}$.
\end{enumerate}
\end{theorem}
The following result is an immediate consequence of the previous theorem together with Theorem~\ref{thm_main_dim} and Theorem~\ref{thm:long_range_cluster}.

\begin{corollary}\label{cor:vm}
For each $G \in \{ \mathcal G_\mathrm{SGil},\mathcal G_\mathrm{Del}, \mathcal G_\mathrm{Gab} \}$, the following holds:
\begin{enumerate}
    \item for~$d=2$, we have~$\mathscr I_e(G) = \{\mu_0,\mu_1\}$ almost surely;
    \item for~$d \ge 3$, we have~$\mathscr I_e(G) =  \{\mu_\alpha: \alpha \in [0,1]\}$ almost surely.
\end{enumerate}
For the long-range percolation cluster~$\mathcal C_\infty$, with the same assumptions as Theorem~\ref{thm:long_range_cluster}, the following holds:
\begin{enumerate}
    \item for~$d=1$ with~$s=2$, and for~$d=2$ with~$s \ge 4$, we have~$\mathscr I_e(\mathcal C_\infty) = \{ \mu_0, \mu_1 \}$ almost surely;
    \item for~$d = 1$ with~$s \in (1,2)$, and for~$d \ge 2$ with~$s \in (d, d+2)$, we have~$\mathscr I_e(\mathcal C_\infty) =  \{\mu_\alpha: \alpha \in [0,1]\}$ almost surely.
\end{enumerate}
\end{corollary}

\subsection{Organization of the paper and proof strategies}
We prove Theorem~\ref{thm_unimodularity_FCP} in Section~\ref{ss_finite_coll_prop}. The proof relies on a mass transport argument that allows us to show that the expected number of collisions is bounded above by the expression appearing in the assumption, whose expectation is finite. Consequently, the expected number of collisions is finite, which implies the finite collision property.

We prove Theorem~\ref{thm_main_dim} in Section~\ref{ss_rgg_et_all}. For this, we transfer the graph to its ``unimodular version'', introduced in Section~\ref{ss_mass_transport_principle}. The transfer lemma, Lemma~\ref{lemma_translate_as}, shows that proving the result for the ``unimodular version'' is equivalent to proving it for the original version. We also prove that the degree of the root is sufficiently controlled, Lemma~\ref{lemma_fin_deg_to_unim}, which allows us to apply the unimodular results. 

We treat dimension 2 in Section~\ref{ss_dim_2}, the result is a direct consequence of Theorem~\ref{thm_unimodularity_ICP}. To verify the assumptions, we use recurrence results for~$\mathcal G_\mathrm{Del}$ and~$\mathcal G_\mathrm{Gab}$ from the literature and prove recurrence of~$\mathcal G_\mathrm{SGil}$ in Section~\ref{ss_recurrence_SGil}. In the high-dimensional case, Section~\ref{ss_dim_3}, the main tool is the application of Theorem \ref{thm_unimodularity_FCP}. The control obtained on the degree of the root reduces the verification of the assumptions of the unimodular result to the computation of effective resistances in the graph, Lemma~\ref{lemma_r_eff_Gil_Gab}. To do so, in Section~\ref{ss_comparison_Bernoulli}, we exploit the deep relation between the geometric graphs considered and Bernoulli percolation. We introduce a general scheme to compare the effective resistance of a \textit{well-behaved} random graph with vertices in $\mathbb R^d$ and the effective resistance in the infinite connected cluster Bernoulli percolation, Lemma~\ref{lem:comparison_effectiveresistance_percolation_domination}. Then, we only need control over the tails of the effective resistance in Bernoulli percolation, for which we apply already known heat kernel estimates of~\cite{barlow2004random}. 

Finally, we prove Theorem~\ref{thm:long_range_cluster} in Section~\ref{ss_long_range_perc}. We consider the regimes in which an infinite cluster exists and show that, conditioned on the origin belonging to the infinite cluster, the infinite cluster rooted at the origin is unimodular. In Section~\ref{ss_reg_long_rec}, we treat the recurrent regimes and verify the assumptions of Theorem~\ref{thm_unimodularity_ICP}. In Section~\ref{ss_reg_long_tran}, we consider the regimes for which heat kernel estimates from~\cite{Sly12} are available, allowing us to verify the assumptions of Theorem~\ref{thm_unimodularity_FCP}.

It is worth noting that Gaussian heat kernel estimates alone are not sufficient to establish the finite or the infinite collision property. Additional assumptions on volume growth allow one to derive stronger conclusions and, in particular, obtain information about recurrence and transience. However, even these ingredients are not sufficient. As we discuss in Section~\ref{ss_collision_properties}, recurrence and transience alone do not determine the collision properties of a graph. 

\section{Preliminaries}
We write~$\NN = \{1,2,3,\dots\}$ and~$\NN_0 = \NN \cup \{0\}$. Throughout this section, we consider only unoriented, locally finite, connected graphs with countable sets of vertices with no loops and with at most one edge between any pair of vertices. Given a graph~$G = (V,E)$, we use~$x,y,z,\dots$ to denote vertices and write~$x \sim y$ when~$\{x,y\} \in E$, and~$\deg(x)$ for the degree of the vertex~$x$. For each~$x,y \in V$ we denote by~$d_G(x,y)$ the graph distance between~$x$ and~$y$, that is, the length of the shortest path in~$G$ connecting~$x$ and~$y$. Given a set or event~$A$, we denote by~$\mathds{1}_A$ its indicator function and by~$|A|$ its cardinality.

\subsection{Random walks and electric networks}\label{ss_rw_and_en}
Given a locally finite graph~$G$, a discrete-time \emph{simple} random walk~$(X_n)_{n \ge 0}$ on~$G$ is the discrete-time Markov chain with state space~$V$ and transition probabilities
\begin{equation}\label{eq_p_transition_srw}
p(x,y) = \deg(x)^{-1} \cdot \mathds 1_{ \{ x \sim y \} }, \quad x,y \in V,    
\end{equation}
We denote by~$p_n$ the~$n$-step transition function of the random walk. The \emph{Green's function} $\mathsf{G}(x,y)$ is defined as the average number of visits to $y$ starting from $x$, formally
\begin{equation*}
    \mathsf{G}(x,y) \eqdef \mathbb E \left. \left[\sum_{n\ge 0} 1_{\{X_n = y\}} \right|  X_0 = x \right]  = \sum_{n\ge 0} p_n(x,y), \quad x,y \in V.
\end{equation*}
Given~$x \in V$, we say that \emph{the graph~$G$ is recurrent} if 
\[
\mathbb P( \exists n \ge 1: X_n = x \mid X_0 = x )=1.
\]
Under the assumption of connectedness, this property is independent of the starting vertex~$x$. On the contrary, if the probability in the last display is strictly less than one, we say that~$G$ is transient. Using the Strong Markov property, we have that in a locally finite, connected graph, 
\begin{align*}
&G \text{ is recurrent } \hspace{-0.5mm} \Leftrightarrow \; \mathsf G(x,y) = \infty \text{ for all } x,y \in V, \\[1mm]
&G \text{ is transient }  \Leftrightarrow \; \mathsf G(x,y) < \infty \text{ for all } x,y \in V.
\end{align*}

To apply the unimodular results presented in the introduction, we require a criterion for recurrence as well as an upper bound on~$\mathsf G$. For this, we use the deep and useful connection between random walks and electrical network theory. For a comprehensive treatment of this theory, we refer the reader to Chapter 2 of~\cite{LyonsPeres17} and Section~2.5 in~\cite{Nachmias20}. Here, we introduce only the main objects and results needed for our purposes. 

% For \(x\in G\), define the first return time to \(x\) by
% \[
% T_x^+ \eqdef \inf\{n\ge 1 : X_n=x\}.
% \]

% We say that the simple random walk on \(G\) is \emph{recurrent} if
% \[
% \mathbb P_x(T_x^+<\infty)=1
% \]
% for some, equivalently every, \(x\in G\).
% We say that the simple random walk on \(G\) is \emph{transient} if
% \[
% \mathbb P_x(T_x^+<\infty)<1
% \]
% for some, equivalently every, \(x\in G\). For our simple random walk in a connected $G$, the definition does not depend on the initial vertex $x$.

Let $\overrightarrow{E}$ denote the set of oriented edges of $G$, we denote by~$\overrightarrow{xy}$ the oriented edge with endpoints~$x,y$. A \emph{unitary flow} $\theta:\overrightarrow{E}\to\mathbb{R}$ from a vertex $x\in V$ to infinity is a function satisfying:
\begin{enumerate}
    \item \textbf{Anti-symmetry:} $\theta(\overrightarrow{yz})=-\theta(\overrightarrow{zy})$ for every edge $\{y,z\}\in E$.
    \item \textbf{Unit source at $x$:} $\displaystyle \sum_{y\sim x}\theta(\overrightarrow{xy})=1$.
    \item \textbf{Flow conservation off $x$:} $\displaystyle \sum_{z\sim y}\theta(\overrightarrow{yz})=0$ for all $y\in V\setminus\{x\}$.
\end{enumerate}
The \emph{effective resistance} from a vertex $x \in V$ to infinity is defined by 
\begin{equation*}
    \mathcal{R}_\mathrm{eff}(x\longleftrightarrow \infty;G) \eqdef \inf \left\{ \frac{1}{2} \cdot \sum_{ \overrightarrow{xy} \in \overrightarrow{E} }  \big( \theta(\overrightarrow{xy}) \big)^2 : \theta \text{ is a unitary flow from } x \text{ to infinity} \right\}
\end{equation*}
Using hitting probabilities, one can relate the effective resistance to return probabilities. As a consequence, one obtains the following relationship between the Green's function and the effective resistance:
\begin{equation}\label{eq_Green_and_Resistance}
    \mathsf{G}(x,x) = \mathrm{deg}(x) \cdot \mathcal{R}_\mathrm{eff}(x\longleftrightarrow \infty;G) \quad \text{ for all } x \in V.
\end{equation}
One can therefore obtain bounds on~$\mathsf G(x,x)$ by constructing flows and estimating their energy. For our purposes, it will be more convenient to construct flows starting from specific vertices. The following inequality, a consequence of the triangle inequality for effective resistances, allows us to move the source of the flow:
\begin{equation}\label{ineq_R_eff_distance_G}
\mathcal{R}_\mathrm{eff}(x\longleftrightarrow \infty;G) \leq d_G(x,y) + \mathcal{R}_\mathrm{eff}(y \longleftrightarrow \infty;G) \quad \text{ for all } x,y \in V.
\end{equation}

We now state a criterion for recurrence. Given~$x \in V$, a set $\Pi \subset E$ is called an \emph{edge-cutset} that separates $x \in V$ from infinity if the removal of $\Pi$ leaves $x$ in a finite connected component.
\begin{lemma}[Nash-Williams]\label{thm:infinite_nash_williams}
    Let $\{\Pi_n\}_{n \ge 1}$ be a disjoint family of edge-cutsets that separate $x$ from infinity. If 
    \[
    \sum_{n \ge 1} \frac{1}{|\Pi_n|} = \infty
    \]
    then~$G$ is recurrent.
\end{lemma}

\subsection{Finite and infinite collision property}\label{ss_collision_properties}
In this section, we discuss the collision properties of graphs for discrete- and continuous-time walkers. \\

\noindent \textbf{Discrete-time collisions.} Let~$(X_n)_{n \ge 0}$ and~$(Y_n)_{n \ge 0}$ be independent discrete-time simple random walks on~$G$. We say that~$G$ has the \emph{discrete infinite collision property} if
\begin{equation}\label{def_inf_coll_prop_disc}
    \PP( |\{ n \in \NN: X_n = Y_n \}| = \infty ) = 1
\end{equation}
for any common starting positions. We say that~$G$ has the \emph{discrete finite collision property} if
\begin{equation}\label{def_fin_coll_prop_disc}
    \PP( |\{ n \in \NN: X_n = Y_n \}| < \infty ) = 1
\end{equation}
for any common starting positions. We stress that these two properties do not, in general, form a dichotomy. Indeed, there exist graphs for which
$$ 0 < \PP( |\{ n \in \NN: X_n = Y_n \}| = \infty ) < 1.$$
Nevertheless, it follows from the proof of Proposition~2.1 in~\cite{BarlowPeresSousi2012} that if~$(X_n)_{n \ge 0}$ has a trivial tail~$\sigma$-algebra, then the probability in the display above is either 0 or 1. Hence, under this assumption, the two properties do form a dichotomy.\\

\noindent \textbf{Continuous-time collisions.} A \emph{continuous-time simple random walk} on~$G$ is a continuous-time Markov chain with state space~$V$ and transition-rate matrix~$p$ from~\eqref{eq_p_transition_srw}. Motivated by applications to the voter model via its duality with coalescing random walks (see the Introduction), we consider a continuous-time analogue of the collision properties above. 

Let~$(X^\mathrm{cont}_t)_{t \ge 0}$ and~$(Y^\mathrm{cont}_t)_{t \ge 0}$ be independent continuous-time simple random walks. We say that~$G$ has the \emph{continuous infinite collision property} if 
\begin{equation}\label{def_inf_coll_prop_cont}
    \PP( |\{ t \ge 0: X^\mathrm{cont}_{t^-} \neq Y^\mathrm{cont}_{t^-},\; X^\mathrm{cont}_t = Y^\mathrm{cont}_t \}| = \infty ) = 1
\end{equation}
for any common starting positions. We say that~$G$ has the \emph{continuous finite collision property} if
\begin{equation}\label{def_fin_coll_prop_cont}
    \PP( |\{ t \ge 0: X^\mathrm{cont}_{t^-} \neq Y^\mathrm{cont}_{t^-},\; X^\mathrm{cont}_t = Y^\mathrm{cont}_t \}| < \infty ) = 1
\end{equation}
for any common starting positions.

Remark~1.3 of~\cite{HutchcroftPeres2015} notes that on non-bipartite graphs, the discrete infinite collision property implies that~\eqref{def_inf_coll_prop_disc} also holds for arbitrary (not necessarily common) starting positions of the random walks. The same observation applies to the discrete finite collision property. An analogous statement holds for the continuous collision properties on any graph, since the continuous-time setting eliminates periodicity effects.

Because of this discrete-time dependence, the discrete and continuous collision properties may, in principle, differ on some graphs. In fact, Chapter~3.1 of the PhD thesis~\cite{Montgomery13} provides an example of a transient graph with unbounded degree that has the continuous infinite collision property, while the probability that two discrete-time simple random walks started from a given vertex collide infinitely many times is strictly smaller than~1.  

Nevertheless, the discrete and continuous collision properties coincide for (vertex-)transitive graphs. Indeed, in this setting, the infinite (respectively finite) collision property, whether in discrete or continuous time, is equivalent to recurrence (respectively transience) of the underlying graph, which is independent of the time parametrization. Beyond the transitive setting,~\cite{Chen08} established the equivalence for quasi-transitive graphs with subexponential growth. The subexponential growth assumption was later removed in Chapter~3.4 of~\cite{Montgomery13}. 

Motivated by the previous discussion, the following question arises naturally. \\[-2mm]

\noindent \textbf{Open question: } Do the discrete and continuous collision properties coincide for simple random walks on bounded-degree, connected graphs?

\subsection{Unimodularity}\label{ss_unimodularity}
In this section, we introduce unimodular random rooted graphs, review several known results, and prove a new observation in Section~\ref{ss_finite_coll_prop}. These will be used throughout the subsequent sections. Given a graph~$G$, a vertex~$x$ and~$n \in \NN$, we define the ball in~$G$ of center~$x$ and radius~$n$ by 
\begin{equation*}
B_G(x,n) = \{y \in V: d_G(x,y) \leq n\}.
\end{equation*}

\noindent \textbf{Unimodular random rooted graphs.} Given a graph~$G$, a~\emph{rooted graph} is a pair~$(G, \rho)$ where~$\rho \in V$ is designated as the root vertex. An isomorphism between two rooted graphs is a graph isomorphism that also maps the roots of the graphs. Let~$\mathcal G_\bullet$ denote the set of isomorphism classes of locally finite, connected rooted graphs. Whenever the context is clear, we use the notation~$(G, \rho)$ to denote either an element of~$\mathcal{G}_\bullet$ or a specific representative of its isomorphism class. Given~$(G,\rho)$, simple random walks on~$G$ are well defined at the level of isomorphism classes of rooted graphs, since their law does not depend on the chosen representative.

We endow~$\mathcal G_\bullet$ with the distance
$$ d_\mathrm{loc}\big( (G_1,\rho_1),(G_2,\rho_2) \big) := \inf \left\{ \frac{1}{n+1}: n \in \mathbb N_0 \text{ and } \big(B_{G_1}(\rho_1,n),\rho_1)\big) \simeq \big(B_{G_2}(\rho_2,n),\rho_2)  \big) \right\},$$
where~$\simeq$ stands for the rooted graph equivalence. With this topology~$\mathcal G_\bullet$ is a Polish metric space, that is, a separable and complete metric space.

Similarly to the space~$\mathcal G_{\bullet}$, we define the space~$\mathcal G_{\bullet \bullet}$ of isomorphism classes of locally finite, connected \emph{bi-rooted graphs}~$(G,x,y)$, which are graphs with two distinguished ordered points. By equipping this set with suitable variants of the distance~$d_\mathrm{loc}$, we obtain a Polish metric space with the corresponding induced topology.

A \emph{random rooted graph}~$(G,\rho)$ is a random variable taking values in~$\mathcal G_\bullet$. We use~$\mathbf P$ and~$\mathbf E$ for the probability and expectation for a random rooted graph. The random rooted graph~$(G,\rho)$ is called \emph{unimodular} if it satisfies the~\emph{Mass-Transport Principle:} for every Borel function~$f:\mathcal G_{\bullet \bullet} \to [0,\infty]$,
\begin{equation}\label{Mass_Transport_Principle}
    \mathbf E \Big[\sum_{v \in V} f(G,\rho,v) \Big] = \mathbf E \Big[\sum_{u \in V} f(G,u,\rho) \Big].
\end{equation}
In words,
\begin{center}
    \emph{Expected mass out equals expected mass in.}
\end{center}
The Mass-Transport Principle was first introduced by Häggström in~\cite{Haggstrom1997} for the study of percolation on Cayley graphs, and later developed further in~\cite{AldousLyons2007}. For an expository treatment of the subject, we refer to the latter reference. \\

\noindent \textbf{Reversible random rooted graphs.} We now briefly introduce reversibility and review its connection with unimodularity, as proved in Proposition~2.5 of~\cite{BenjaminiCurien2012}. This relation allows us to transfer results from reversible random rooted graphs to unimodular ones.

Let~$(G,\rho)$ be a random rooted graph. Conditionally on~$(G,\rho)$, let~$(X_n)_{n \ge 0}$ be a simple random walk on~$G$ starting from~$\rho$. The graph~$(G,\rho)$ is called \emph{reversible} if
$$ (G,X_0,X_1) = (G,X_1,X_0) \quad \text{in distribution.} $$

Assume that $\mathbf E[\deg(\rho)]<\infty$. We say that $(\widetilde G,\widetilde \rho)$ is the \emph{degree-biased version} of $(G,\rho)$ if, for every non-negative Borel function $f:\mathcal G_\bullet\to\mathbb R$,
\[
\mathbf E[f(\widetilde G,\widetilde \rho)]
=
\frac{\mathbf E[f(G,\rho)\deg(\rho)]}{\mathbf E[\deg(\rho)]}.
\]

\begin{proposition}[\cite{BenjaminiCurien2012}]\label{prop_rev_unimod}
Let~$(G,\rho)$ be a unimodular random rooted graph with~$\mathbf E[\deg(\rho)]<\infty$. Then biasing its law by~$\deg(\rho)$ yields a reversible random rooted graph. Conversely, if~$(G,\rho)$ is reversible, then biasing its law by~$\deg(\rho)^{-1}$ yields a unimodular random rooted graph.
\end{proposition}

\subsubsection{Infinite collision property}
In this section, we prove Theorem~\ref{thm_unimodularity_ICP}. The discrete-time part was proved in Theorem~1.5 of~\cite{HutchcroftPeres2015}. The continuous-time analogue is treated in Theorem~3.2 of the same article, where the result is established for random walks with symmetric jump rates along edges. This does not apply directly to our setting, since here the jump rates are given by the non-symmetric function~$p$ from~\eqref{eq_p_transition_srw}. Nevertheless, the same arguments can be adapted to prove the continuous infinite collision property in our setting. For completeness, we include the proof.

\begin{proof}[Proof of Theorem~\ref{thm_unimodularity_ICP}]
We begin by establishing a continuous-time reversibility property of the unimodular graph~$(G,\rho)$, namely that for every~$t \ge 0$,
\begin{equation}\label{eq_reversibility_cont}
(G,\rho,X^\mathrm{cont}_t) = (G,X^\mathrm{cont}_t,\rho) \quad \text{ in distribution.}  
\end{equation}

Fix~$t \ge 0$ and consider the following functions,
$$ \begin{array}{cccc}
    \mathsf p_t:&  \mathcal G_{\bullet \bullet} & \to & [0,1] \\[2mm]
    & (G,x,y) & \mapsto & p^\mathrm{cont}_t(x,y)
\end{array} \quad \text{ and } \quad \begin{array}{cccc}
    \nu: & \mathcal G_{\bullet } & \to & (0,\infty) \\[2mm]
    &(G,x) & \mapsto & \deg_G(x)/\mathbf{E}[\deg(\rho)],
\end{array}$$
where~$p^\mathrm{cont}_t(\cdot,\cdot)$ denotes the transition probability of~$(X^\mathrm{cont}_t)_{t \ge 0}$ on~$G$ at time~$t$. Note that for each graph~$G$, the measure~$\nu(G,\cdot)$ is reversible for the discrete-time random walk on~$G$ with transition function~$\mathsf p_t(G, \cdot , \cdot )$, and moreover~$\mathbf E[ \nu(G,\rho)] = 1$. Therefore, Theorem~4.1 of~\cite{AldousLyons2007}, yields~\eqref{eq_reversibility_cont}.

We now use the recurrence of~$G$. An argument based on the exponential holding times shows that~$(X^\mathrm{cont}_n)_{n \in \NN}$ is recurrent a.s. It then follows from Theorem~3.1 of~\cite{HutchcroftPeres2015} that reversibility~\eqref{eq_reversibility_cont} with~$t = 1$ and recurrence imply that
$$ |\{ n \in \NN: X^\mathrm{cont}_n = Y^\mathrm{cont}_n \}| = \infty \quad \text{almost surely.} $$
Finally, applying once again an exponential holding time argument, we obtain~\eqref{def_inf_coll_prop_cont}.
\end{proof}

\subsubsection{Finite collision property}\label{ss_finite_coll_prop}
In this section, we prove Theorem~\ref{thm_unimodularity_FCP}. We first recall the following standard fact: given a locally finite, connected graph~$G$, 
\begin{equation}\label{eq_standar_cont_disc}
\int^\infty_0 p^\mathrm{cont}_t(x,x) \; \mathrm{d}t = \sum_{n \ge 0}p_n(x,x), \quad \text{for all } x \in V.
\end{equation}

\begin{proof}[Proof of Theorem~\ref{thm_unimodularity_FCP}] We assume that~$(G,\rho)$ is unimodular and that~$\mathbf{E} [\deg(\rho) \cdot \mathsf G(\rho)] < \infty$. We denote by~$\mathbf P^G$ the joint law of two independent discrete-time random walks and two independent continuous-time random walks on~$G$, and by~$\mathbf E^G$ its expectation operator. We define
$$ Z_\mathrm{disc}(x) := |\{n \in \NN: X_n = Y_n\}| \quad \text{and} \quad Z_\mathrm{cont}(x) := |\{ t \ge 0: X^\mathrm{cont}_{t^-} \neq Y^\mathrm{cont}_{t^-}, \; X^\mathrm{cont}_t = Y^\mathrm{cont}_t \}| $$
with the random walks started both from~$x \in V(G)$. 

We start by proving that~$G$ has the discrete finite collision property almost surely. Using the reversibility of the random walk, we have 
\begin{equation}\label{eq_pre_apply_MTP}
    \deg(x) \cdot \mathbf E^G [Z_\mathrm{disc}(x)] = \sum^\infty_{n = 0} \sum_{v \in V(G)} p_n(x,v) \cdot p_n(v,x) \cdot \deg(v),
\end{equation}
Consider the mass transport function
$$ f(G,u,v) = p_n(u,v) \cdot p_n(v,u) \cdot \deg(v) $$
Each vertex~$u$ sends a total mass of~$\sum_{v \in V(G)}p_n(u,v) p_n(v,u) \deg(v)$, while, each vertex~$v$ receives a total mass of
$$ \sum_{u \in V(G)} f(G,u,v) = \deg(v) \cdot p_{2n}(v,v). $$
We apply the Mass-Transport Principle~\eqref{Mass_Transport_Principle} in~\eqref{eq_pre_apply_MTP} to obtain that 
$$ \mathbf{E} \big[ \deg(\rho) \cdot \mathbf E^G[Z_\mathrm{disc}(\rho) ] \big] = \mathbf{E} \big[ \deg(\rho) \cdot \sum^\infty_{n = 0} p_{2n}(\rho,\rho) \big] \leq \mathbf E[ \deg(\rho) \cdot \mathsf G(\rho,\rho) ]. $$
Hence, the hypothesis implies that~$Z_\mathrm{disc}(\rho)<\infty$~$\mathbf P$-a.s., which, combined with the connectedness of~$G$, yields that~$G$ has the discrete finite collision property almost surely.

We now prove the continuous analogue. To this end, we use Wald's equation and then Fubini's theorem to see that
$$ 2 \cdot \mathbf E^G [ Z_\mathrm{cont}(x)] = \EE^G[ \mathrm{Leb}(t \ge 0: X^\mathrm{cont}_t = Y^\mathrm{cont}_t) ] = \int^\infty_0 \sum_{v \in V(G)} p^\mathrm{cont}_t(x,v) \cdot p^\mathrm{cont}_t(v,x) \cdot \frac{\deg(v)}{\deg(x)}, $$
where Leb stands for the Lebesgue measure on~$\RR$. Applying the Mass–Transport Principle in the same way as in the discrete case, we obtain that
\begin{equation*}
    \mathbf E \big[ 2 \deg(\rho) \cdot \mathbf E^G[Z_\mathrm{cont}(\rho)] \big] = \mathbf E \big[ \deg(\rho) \cdot \int^\infty_0 p^\mathrm{cont}_{2t}(\rho,\rho) \; \mathrm{d}t \big].
\end{equation*}
Finally, changing variables~$t \to 2t$, plugging~\eqref{eq_standar_cont_disc} into the previous display and arguing as in the discrete setting, we conclude that~$G$ has the continuous finite collision property almost surely.
\end{proof}

\subsubsection{Tail~$\sigma$-algebra}
Theorem~3.2 of~\cite{BenjaminiCurien2012} establishes that the tail~$\sigma$-algebra of a discrete-time simple random walk on a reversible graph with subexponential growth is trivial almost surely. We now state an analogous result for unimodular graphs, using the connection between reversibility and unimodularity given in Proposition~\ref{prop_rev_unimod}, and include the proof for completeness.
\begin{proposition}\label{thm_Liouville}\cite{BenjaminiCurien2012}
    Let~$(G,\rho)$ be a unimodular random rooted graph satisfying
    \begin{equation}\label{limit_log_ball}
        \lim_{n \to \infty} n^{-1} \cdot \mathbf E \big[ \deg(\rho) \cdot \log \big( B_G(\rho,n) \big) \big] = 0.
    \end{equation}
    Then, the simple random walk on~$G$ started from any vertex of the graph has a trivial tail~$\sigma$-algebra almost surely.
\end{proposition}
\begin{proof}
It is easy to see that~\eqref{limit_log_ball} implies~$\mathbf E[\deg(\rho)]<\infty$. Let~$(\tilde G,\tilde \rho)$ denote the reversible random rooted graph obtained by biasing~$(G,\rho)$ with~$\deg(\rho)$. By direct application of \cite[Theorem~3.2]{BenjaminiCurien2012}, it follows from~\eqref{limit_log_ball} that a discrete-time random walk on~$\tilde G$ started from~$\tilde \rho$ has a trivial tail $\sigma$-algebra almost surely. Since this holds almost surely, the same property transfers to~$(G,\rho)$. Moreover, due to the connectedness of~$G$, the triviality of the tail $\sigma$-algebra holds for random walks starting from any vertex of the graph. 
\end{proof}

% \subsubsection{Finite intersection of trajectories}
% {\color{blue}
% I am not sure if we should include this part. It might be nice. How many times do the trajectories of two independent simple random walks~$(X_n)_{n \ge 0}, (Y_n)_{n \ge 0}$ started from the same location intersect on a unimodular graph? Let define
% $$ N_\mathrm{int}: = |\{ (n,m): X_n = Y_m \}| $$
% Using the reversibility of the random walk, we have
% $$ \deg(x) \cdot \EE^G N_\mathrm{int}  = \sum_{n \ge 0} \sum_{m \ge 0} \sum_{y \in V} p_n(x,y) \cdot p_m(y,x) \cdot \deg(y). $$
% We now use the Mass-Transport Principle to obtain,
% $$ \mathbf{E}[ \deg(\rho) \cdot \EE^G  N_\mathrm{int}  ] = \mathbf{E} \left[ \deg(\rho) \cdot \sum_{n \ge 0} \sum_{m \ge 0} p_{n+m}(\rho,\rho) \right]. $$
% A good upper bound on the heat kernels again would imply that the number of intersections is finite almost surely.
% }

\section{Random graphs generated by point processes in $\RR^d$}\label{ss_rgg_et_all}
In this section, we apply the results on unimodular graphs from Section~\ref{ss_unimodularity} to study the collision properties of random walks on random graphs generated by a Poisson point process under geometric proximity conditions. We write~$\norm{ \cdot }$ for the Euclidean norm and~$B_\mathrm{Eucl}(x,r)$ for the Euclidean ball of center at~$x \in \RR^d$ and radius~$r \ge 0$; and~$\mathrm{Leb}$ for the Lebesgue measure on~$\mathbb{R}^d$. Throughout this section, we assume that~$d \ge 2$. \\

\noindent \textbf{Poisson point process and Mecke equation.} Let~$\mathbb S$ denote the space of all locally finite sets of points~$\mathcal P \subseteq \RR^d$ and endow~$\mathbb S$ with the smallest~$\sigma$-algebra such that the maps
$$ \mathcal P \mapsto |\mathcal P \cap A|  $$
are measurable for all Borel sets~$A \subseteq \RR^d$. A random variable~$\mathcal P$ with values in~$\mathbb S$ is called a point process. Let $\mathcal P$ be a homogeneous Poisson point process with intensity $\lambda$. Given~$\lambda > 0$, a (homogeneous) Poisson point process with intensity~$\lambda$ is a point process~$\mathcal P$ such that
\begin{itemize}
    \item Poisson distribution. The random variable~$|\mathcal P \cap A|$ is Poisson distributed with parameter~$\lambda \cdot \mathrm{Leb}(A)$ for any Borel~$A \subseteq \RR^d$.
    \item Independence. The random variables~$|\mathcal P \cap A_1|,\dots |\mathcal P \cap A_n|$ are independent for any pair-wise disjoint Borel sets~$A_1,\dots,A_n \subseteq \RR^d$.
\end{itemize}

We denote by~$\mathbb{P}_0$ the Palm distribution of the point process~$\mathcal P$; see Section~9 of~\cite{LastPenrose2017} for its definition. When~$\mathcal P$ is a Poisson point process, the Mecke–Slivnyak theorem (ibid.) identifies~$\mathbb{P}_0$ with the law of~$\mathcal P \cup \{0\}$, that is, the process obtained by adding a point at the origin. We will use the following form of Mecke equation: Let~$f: \RR^d \times \mathbb S \to [0,\infty)$ be measurable and~$\mathcal P$ be a homogeneous Poisson point process. Then
$$ \EE \left[ \sum_{x \in \mathcal P}f(x,\mathcal P) \right] = \lambda \cdot \int_{\mathbb R^d}\EE[ f(x,\mathcal P \cup \{x\} ) ] \; \mathrm{d}x $$

We now define the graphs on which we study collision properties. Throughout the remainder of this section, we let~$\mathcal P$ denote a Poisson point process on~$\mathbb{R}^d$ with intensity~1. The results extend straightforwardly to any intensity~$\lambda$. \\

\noindent \textbf{The supercritical Gilbert graph.} Given the radius of influence~$r > 0$, the \emph{Gilbert graph}, also known in the literature as the \emph{random geometric graph},~$\mathcal G_\mathrm{Gil} = \mathcal{G}_\mathrm{Gil}(r,d) = (V,E)$ is defined by setting
$$ V = \mathcal P \quad \text{ and } \quad  E = \{ \{u,v\} \subseteq V: \norm{ u - v} < r, \; u \neq v \}, $$
It is a well-known fact from continuum percolation theory (see, for instance~\cite{Penrose2003}) that for all dimensions~$d \ge 2$, there exists a critical parameter~$r_c = r_c(d) > 0$ such that, if~$r >r_c$ the random graph~$ \mathcal G_\mathrm{Gil}$ has a unique infinite connected component~$\mathcal G_\mathrm{SGil} =\mathcal G_\mathrm{SGil}(r,d)$ almost surely. \\

% \noindent \textbf{Voronoi 1-skeleton graph.} For any~$x \in \mathcal P$, we define the Voronoi cell of~$x$ as the set of points of~$\RR^d$ which are closer to~$x$ than to any other point of~$\mathcal P$: $$\mathrm{Vor}_\mathcal{P}(x) := \{ u \in \RR^d: \norm{ u - x } \leq \norm{ u -y }, \forall y \in \mathcal P \}.$$
% The collection~$\{\mathrm{Vor}(x): x \in \mathcal P\}$ tessellates~$\RR^d$ into convex polyhedra. The~\emph{Voronoi skeleton graph}~$\mathcal{G}_\mathrm{Vor}= (V,E)$ is defined by setting
% \begin{align*}
%     & V = \text{Points on the boundaries of the Voronoi cells (intersection of } d+1 \text{ cells)}, \\[2mm]
%     & E = \text{One-dimensional edges of the Voronoi cells (intersections of } d \text{ cells).}
% \end{align*}
% For the point process~$\mathcal P$ in~$\RR^d$, we have that~$\mathcal G_\mathrm{Vor}$ is a~$(d+1)$-regular graph almost surely for any~$d \ge 1$. \\

\noindent \textbf{Delaunay triangulation.} For any~$x \in \mathcal P$, we define the Voronoi cell of~$x$ as the set of points of~$\RR^d$ which are closer to~$x$ than to any other point of~$\mathcal P$: $$\mathrm{Vor}_\mathcal{P}(x) := \{ u \in \RR^d: \norm{ u - x } \leq \norm{ u -y }, \forall y \in \mathcal P \}.$$
The collection~$\{\mathrm{Vor}_\mathcal{P}(x): x \in \mathcal P\}$ tessellates~$\RR^d$ into convex polyhedra. The \emph{Delaunay graph}~$\mathcal G_\mathrm{Del} = (V,E)$ is defined by setting:
$$ V = \mathcal P, \quad E = \{ \{x,y\}: \mathrm{Vor}_\mathcal{P}(x) \text{ and } \mathrm{Vor}_\mathcal{P}(y) \text{ share a } (d-1)\text{-dimensional face.} \}$$
For the point process~$\mathcal P$ in~$\mathbb{R}^d$, the Delaunay triangulation is well defined, in the sense that it is almost surely unique. This triangulation admits the following equivalent characterization: a~$d$-simplex with vertices in~$\mathcal P$ belongs to~$\mathcal G_\mathrm{Del}$ if and only if its circumscribed ball contains no points of~$\mathcal P$ in its interior. \\

\noindent \textbf{Gabriel graph.} The \emph{Gabriel graph}~$\mathcal G_\mathrm{Gab} = (V,E)$ is defined by setting 
$$ V = \mathcal P, \quad E = \{ \{x,y\}: \text{the ball of diameter } [x,y] \text{ contains no points of } \mathcal P \text{ in its interior.} \} $$
Using the equivalent characterization of the Delaunay graph, it follows that~$\mathcal G_\mathrm{Gab}$ is a subgraph of~$\mathcal G_\mathrm{Del}$. \\

In Section~\ref{ss_mass_transport_principle}, we review how to construct unimodular versions of the random graphs introduced above. We then apply the unimodular results developed in Section~\ref{ss_unimodularity} to prove part~1 of Theorem~\ref{thm_main_dim} in Section~\ref{ss_dim_2}, part~2 in Section~\ref{ss_dim_3} and part~3 in Section~\ref{ss_dim_2_trivial}.

\subsection{Unimodularity for geometric graphs}\label{ss_mass_transport_principle}
In~\cite[Example~9.5]{AldousLyons2007}, the authors show that any random graph whose law is invariant under the isometries of a Euclidean or hyperbolic space admits a unimodular version. We review this construction for a more restricted class of graphs that includes those defined in the previous section and suffices for our purposes. This also provides an opportunity to introduce the notation that will be used in the subsequent sections. \\

\noindent \textbf{Unimodular construction.} Let~$\PP$ be a probability space on which the point process~$\mathcal P$ is defined, and write~$\EE$ for the expectation operator with respect to~$\PP$. We consider locally finite, connected random graphs of the form~$G = \beta(\mathcal P)$ embedded in~$\mathbb{R}^d$, constructed according to deterministic, measurable instructions~$\beta$ that commute with the isometries of~$\mathbb{R}^d$. Note that~$\mathcal G_\mathrm{SGil}$,~$\mathcal G_\mathrm{Del}$ and~$\mathcal G_\mathrm{Gab}$ satisfy these conditions. 

Fix~$A \subset \RR^d$ with~$\mathrm{Leb}(A) \in (0,\infty)$. If~$v(A) := \EE |V(G) \cap A| < \infty $, we define the unimodular random rooted graph~$(G_\mathrm{geo},\rho)$ as follows. First, sample~$G_\mathrm{geo}$ according to the law of~$G$ biased by~$|V(G) \cap A|$. Then choose the root~$\rho$ uniformly among the vertices of~$G_\mathrm{geo}$ that lie in~$A$. The distribution of the resulting random rooted graph~$(G_\mathrm{geo},\rho)$, denoted by~$\mu_\mathrm{geo}$, satisfies
\begin{equation}\label{eq_unimodular_H}
    \mu_\mathrm{geo}(\mathcal A) := v(A)^{-1} \cdot \EE \left[  \sum_{x \in V(G) \cap A}\mathds{1}\{ (G,x) \in \mathcal A \} \right] \quad \text{for every Borel }  \mathcal A \text{ of } \mathcal G_\bullet.
\end{equation}
One can follow the same steps as in Example 9.5 of Section 9 in~\cite{AldousLyons2007} to prove that~$(G_\mathrm{geo},\rho)$ satisfies the Mass-Transport principle~\eqref{Mass_Transport_Principle} and that its distribution~$\mu_\mathrm{geo}$ does not depend on the choice of~$A$. 

We note two consequences of the fact that the construction rule~$\beta$ commutes with the isometries of~$\mathbb{R}^d$. First, the assumption that~$v(A)< \infty$ for some fixed Borel set~$A \subset \RR^d$ with~$\mathrm{Leb}(A) \in (0,\infty)$ is equivalent to the same condition holding for every Borel set with positive finite Lebesgue measure. In particular, it suffices to verify this condition for
\begin{equation}\label{eq_gamma_G_def}
\gamma_G := v([0,1]^d) = \EE |V(G) \cap [0,1]^d|  < \infty.
\end{equation}
Under this assumption, the vertex set~$V(G)$ defines a stationary point process on~$\RR^d$, that is, a point process whose distribution is invariant under spatial shifts. By Proposition~8.2 of~\cite{LastPenrose2017},
$$ v(A) = \EE | V(G) \cap A| = \gamma_G \cdot \mathrm{Leb}(A) \quad  \text{ for every Borel set } A.$$
Hence, if~$\gamma_G = 0$, then~$V(G) = \varnothing$ almost surely. 

Second, the ergodicity of~$\mathcal P$ is inherited by~$V(G)$. Thus, we have the ergodic theorem~\cite[Theorem 8.14]{LastPenrose2017}: let~$\Lambda_n$ be a box of side-length~$n$ centered at the origin. Then, as~$n \to \infty$,  
\begin{equation}\label{eq_ergo_rd_ppp}
    \frac{|V(G) \cap \Lambda_n| }{\mathrm{Leb}(\Lambda_n)} \xrightarrow{ n \to \infty} \gamma_G \quad \text{in } L^1(\PP).
\end{equation}
We now use this ergodic result to obtain a lemma that enables the transfer of properties such as recurrence/transience and collision properties between~$G_{\mathrm{geo}}$ and~$G$. 

\begin{lemma}[Transfer lemma]\label{lemma_translate_as}
Let~$G = \beta(\mathcal P)$ with~$\gamma_G \in (0,\infty)$. For any Borel set~$\mathcal A \text{ of } \mathcal G_\bullet$,
$$ \PP\big((G,x) \in \mathcal A \text{ for all } x \in V(G) \big) = 1 \text{ if and only if } \mu_\mathrm{geo}(\mathcal A) = 1.$$
\end{lemma}
\begin{proof}
Assume~$\PP\big((G,x) \in \mathcal A \text{ for all } x \in V(G) \big) = 1$. Then, by the definition in~\eqref{eq_unimodular_H}, it follows that~$\mu_\mathrm{geo}(\mathcal A) = 1$. Conversely, assume~$\mu_\mathrm{geo}(\mathcal A) = 1$. Since~$\mu_\mathrm{geo}$ does not depend on~$A$, we have that for any~$A \subset \RR^d$ with~$\mathrm{Leb}(A) \in (0,\infty)$:
$$ 0 = \EE \Big[ \sum_{x \in V(G) \cap A}\mathds{1}\{ (G,x) \in \mathcal A^c \} \Big] \ge \PP \big( \{ \exists x \in V(G) \cap A: (G,x) \in \mathcal A^c\} \cap \{|V(G) \cap A| \ge 1\} \big).$$
Then,
$$ \PP \big( \exists x \in V(G) \cap \Lambda_n: (G,x) \in \mathcal A^c \big) \leq \PP( |V(G) \cap \Lambda_n| = 0 ).$$
The probability on the right-hand side tends to~0 as~$n \to \infty$ by the ergodic theorem~\eqref{eq_ergo_rd_ppp} and the fact that~$\gamma_G > 0$. Hence, the converse follows.
\end{proof}

Therefore, to be able to use the unimodular construction and the transfer lemma, we need~$\gamma_G \in (0,\infty)$. We now verify this condition for the graphs under consideration. 

By the Mecke equation, 
$$ \PP_0( 0 \in V(\mathcal G_\mathrm{SGil}) ) = \gamma_{\mathcal G_\mathrm{SGil}} \leq \EE |\mathcal P \cap [0,1]^d| = 1, $$
On the other hand, stationary implies that~$\PP_0( 0 \in V(\mathcal G_\mathrm{SGil}) )>0$ whenever~$r > r_c$. Since the vertex sets of the Delaunay and Gabriel graphs coincide with~$\mathcal P$, we have
$$ \gamma_{\mathcal G_\mathrm{Del}} = \gamma_{\mathcal G_\mathrm{Gab}} = \EE |\mathcal P \cap [0,1]^d| = 1.$$
% Finally, unlike the previous cases, the verification of~$\gamma_{\mathcal{G}_{\mathrm{Vor}}} \in (0,\infty)$ is more involved. For this reason, we simply refer to~\cite[Section~3, equation~(3.9)]{EDELSBRUNNER2017}, where an explicit expression for this constant is given.  

To conclude this section, we establish sufficient control on the degree of the origin under the Palm measure, which will allow us to apply the unimodular results. 
% Recall that the Voronoi skeleton graph is~$(d+1)$-regular.
\begin{lemma}\label{lemma_fin_deg_to_unim}
Fix~$d \ge 2$. For~$G \in \{ \mathcal{G}_\mathrm{SGil}, \mathcal{G}_\mathrm{Del}, \mathcal G_\mathrm{Gab}$\} we have that
\begin{equation}\label{eq_exp_deg_G}
\EE_0 \big[ \exp \big( \theta \deg_G(0) \big) \cdot \mathds{1}\{0 \in V(G)\} \big] < \infty \quad \text{for all } \theta \in \RR.
\end{equation}
\end{lemma}
\begin{proof}
We begin with the random geometric graph. Note that on~$\PP_0$,
$$ \deg_{\mathcal{G}_\mathrm{SGil}}(0) = |\mathcal P \cap B_\mathrm{Eucl}(0,r)|.$$
The latter is a Poisson random variable with parameter~$\mathrm{Leb}(B_\mathrm{Eucl}(0,r))$, and hence~$\eqref{eq_exp_deg_G}$ follows. Since~$\mathcal{G}_\mathrm{Gab}$ is a subgraph of~$\mathcal{G}_\mathrm{Del}$, it suffices to consider the Delaunay graph. Upper bounds on~$\PP_0 (\deg_{\mathcal{G}_\mathrm{Del}}(0) = m)$ are given in Proposition~6 of~\cite{Bonnet20}. In particular, it is shown that there exists~$c > 0$, depending only on~$d$, such that
$$\PP_0 (\deg(0) = m) \leq c^m m^{-\frac{2}{d-1}m} \quad \text{for } m \text{ large enough}.  $$
This clearly implies~\eqref{eq_exp_deg_G}.
\end{proof}
We will repeatedly use the idea of the last proof in the case of geometric graphs: We will transfer calculations on the unimodular version of our graphs to the Palm version.

\subsection{Dimension 2}\label{ss_dim_2}
We divide this section into two parts. Section~\ref{ss_prove_ICP_geo} is devoted to the proof of part~1 of Theorem~\ref{thm_main_dim}, the infinite collision property. The argument relies on the recurrence of~$\mathcal G_\mathrm{SGil}$, established in Section~\ref{ss_recurrence_SGil}, while recurrence for the remaining graphs follows from~\cite{recurrenceVoronoiPPP}.

\subsubsection{The infinite collision property}\label{ss_prove_ICP_geo}
In this section, we prove that each graph~$G \in \{ \mathcal{G}_\mathrm{SGil}, \mathcal{G}_\mathrm{Del},\mathcal{G}_\mathrm{Gab} \}$ has the discrete and continuous infinite collision property.
\begin{proof}[Proof of part one of Theorem~\ref{thm_main_dim}] 
We begin by establishing almost sure recurrence: Theorem~1 of~\cite{recurrenceVoronoiPPP} states that~$\mathcal{G}_{\mathrm{Del}}$ and~$\mathcal{G}_{\mathrm{Gab}}$ are recurrent, while the recurrence of~$\mathcal G_\mathrm{SGil}$ is established in Section~\ref{ss_recurrence_SGil}. 

Fix~$G \in \{ \mathcal{G}_\mathrm{SGil}, \mathcal{G}_\mathrm{Del},\mathcal{G}_\mathrm{Gab} \}$ and recall the unimodular random graph~$(G_\mathrm{geo},\rho)$ from Section~\ref{ss_mass_transport_principle}. Define the~$\mathcal G_\bullet$-Borel sets,
\begin{align*}
& \mathcal A_\mathrm{rec} = \{ (G,\rho): \text{a simple random walk started from } \rho \text{ returns to } \rho \text{ a.s.} \}, \\[2mm]
& \mathcal A_\mathrm{ICP} = \{(G,\rho): \eqref{def_inf_coll_prop_disc} \text{ and } \eqref{def_inf_coll_prop_cont} \text{ hold for the walkers started from } \rho \}.
\end{align*}
These sets will be used in the application of the transfer lemma (Lemma~\ref{lemma_translate_as}). Note that connectedness extends the properties in the sets above from the root to all vertices in the graph. Since~$G$ is recurrent a.s., we can transfer this property to~$G_\mathrm{geo}$. We now turn to the infinite collision property of~$G_{\mathrm{geo}}$. This property will follow once the unimodularity result of Theorem~\ref{thm_unimodularity_ICP} is applied, after which it can be transferred to~$G$. This will complete the proof.

Thus, the remaining step is to verify the integrability of the degree at the root; for this, it is enough to have~$\EE[ \sum_{x \in V \cap [0,1]^d} \deg_G(x) ]  < \infty$. For this, we use the Mecke equation and the isometries of~$G$ to obtain that
$$ \EE \Big[ \sum_{x \in V(G) \cap [0,1]^d} \deg_G(x) \Big] = \EE \Big[ \sum_{x \in \mathcal P} \deg_G(x) \cdot \mathds{1}_{ \{ x \in V(G) \cap [0,1]^d\} } \Big] = \EE_0[\deg_G(0) \cdot \mathds{1} \{0 \in V(G) \} ].$$
Therefore, Lemma~\ref{lemma_fin_deg_to_unim} implies that this expectation is finite.
\end{proof}

\subsubsection{Recurrence of~$\mathcal G_\mathrm{SGil}$}\label{ss_recurrence_SGil}
We now prove the recurrence of the infinite connected component of the supercritical Gilbert graph. To this end, we use the following lemma.
 
\begin{lemma}\label{lem:billingsley}
Let~$\{N_i: i \in \NN\}$ be an i.i.d. family of Poisson random variables with mean~$\lambda > 0$. There exists~$c_\mathrm{pois} = c_\mathrm{pois}(\lambda) > 0$ such that 
\begin{equation}\label{ineq_N_2_Poisson}
\mathbb P \left( \sum^m_{i=1}N^2_i \ge c_\mathrm{pois} \cdot m \right) \leq \frac{1}{m^2} \quad \text{for any } m \in \NN.
\end{equation}
\end{lemma}
\begin{proof}
The proof is an application of Theorem 12.2 in~\cite{Billingsley68}. In particular, inequality~(12.48), presented there for the case of independent and identically distributed random variables, implies the following. Let~$\{\xi_i: i \in \NN\}$ be an i.i.d. family of random variables with
$$  \EE[\xi_1] = 0 \quad \text{and} \quad \EE[\xi^4_1] =: \sigma < \infty.$$
Then, there exists~$K> 0$ such that for any~$r > 0$ and~$m \in \NN$,
$$ \PP \left( \sum^m_{i=1} \xi_i \ge r \right) \leq 4 \sigma \cdot K \cdot \frac{m^2}{r^4}.$$
Hence, to obtain~\eqref{ineq_N_2_Poisson} we apply the last inequality with the choice of
$$ \xi_i = N^2_i - \EE \big[N^2_i \big], \; i \in \NN, \quad r = (c_\mathrm{pois} - \EE \big[N^2_1 \big]) \cdot m \quad \text{and} \quad c_\mathrm{pois} = \EE[N^2_1] + (4 \sigma \cdot K)^{1/4}. $$
\end{proof}

\begin{lemma}
$\mathcal{G}_\mathrm{SGil}$ is recurrent almost surely.
\end{lemma}
\begin{proof}
For each~$n \in \NN$ we define the square~$C_n \subset \RR^2$ of side-length~$2rn$ centered at the origin and the set of edges
$$ \Pi_n = \{ e = \{x,y\} \in E: \; x \in C_n \text{ and } y \notin C_n \} $$
Note that if~$x \in V \cap C_m$ for some~$m \in \NN$, then the sets~$\{\Pi_n: n \geq m \}$ form a disjoint family of edge-cutsets separating~$x$ from infinity. Hence, for the desired result, it is enough to prove that
\begin{equation}\label{eq:enough_sum}
    \sum_{n \ge 1} \mathbb P(|\Pi_n| \ge cn)  < \infty,
\end{equation}
for some~$c > 0$. Indeed, by Borel-Cantelli lemma there exists~$N$ such that
$$ |\Pi_n| < cn \text{ for all } n \ge N \text{ a.s.}$$
Consequently,~$\sum_{n \ge 1} |\Pi_n|^{-1} = \infty$ almost surely. The Nash–Williams lemma, Lemma~\ref{thm:infinite_nash_williams}, completes the proof of the claim.

We now prove~\eqref{eq:enough_sum}. Fix~$n \ge 1$. To obtain an upper bound for~$|\Pi_n|$, we \emph{partition} the annulus~$C_{n+1} \setminus C_{n-1}$ into squares of side-length~$r$. Let~$\mathcal B_1$ be the family of squares contained in~$C_n$ and let~$\mathcal B_2$ be the family of squares lying outside. Note that~$|\mathcal B_1| = 4(2n-1)$ and~$|\mathcal B_2| = 4(2n+1)$ See figure~\ref{fig:RGG_blocks}. We regard the squares as compact sets, so neighboring squares intersect either along a common side or at a corner point. For each pair~$(B,B') \in  \mathcal B_1 \times \mathcal B_2$ with~$B \cap B' \neq \varnothing$, the number of edges connecting them is upper bounded by the product of the counts of independent Poisson points~$(N_B,N_{B'})$ in each square. Hence,
$$ |\Pi_n| \leq \sum_{(B,B') \in \mathcal B_1 \times \mathcal B_2: B \cap B' \neq \varnothing} N_B \cdot N_{B'}.$$
Observe that, in the sum, each inner corner square is counted five times, each outer corner square once, and every other square three times. Using the inequality~$x \cdot y \leq x^2 + y^2$, we have that
$$ |\Pi_n| \leq 5 \cdot \left( \sum_{B \in \mathcal B_1} N^2_B + \sum_{B' \in \mathcal B_2} N^2_{B'} \right) = 5 \cdot \sum^{16n}_{i=1}N^2_i,$$
where~$N_1,\dots,N_{16n}$ are a family of i.i.d. Poisson random variables of intensity~$r^2$. Therefore,~\eqref{eq:enough_sum} with~$c = 80 c_\mathrm{pois}(r^2)$ follows from the above upper bound together with an application of Lemma~\ref{lem:billingsley}.
\begin{figure}[h]
    \centering
    \includegraphics[width=0.65\linewidth]{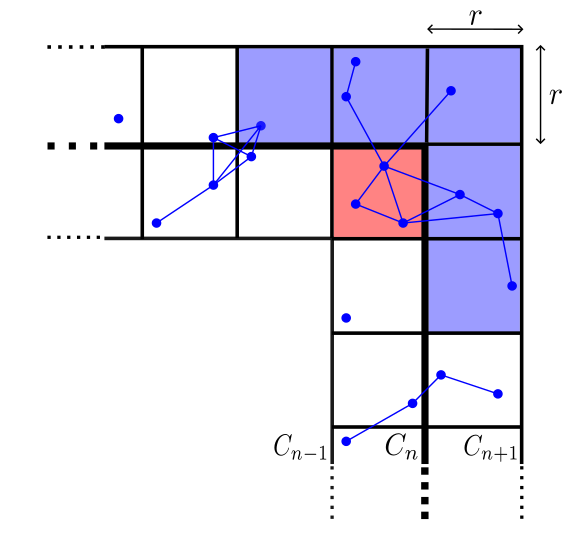}
    \caption{Illustration of a corner of $C_{n+1} \setminus C_{n-1}$ partitioned into~$r$-squares on the inner and outer boundaries of~$C_n$. The number of edges in~$\Pi_n$ with an endpoint in the red region is bounded from above by (number of vertices in the red region)$\times$(number of vertices in the blue region).}
    \label{fig:RGG_blocks}
\end{figure}
\end{proof}

\subsection{Dimension 3 and higher}\label{ss_dim_3}
In this section, we prove part 2 of Theorem~\ref{thm_main_dim}, namely the finite collision property. The section is organized as follows. In Section~\ref{ss_FCP_dim3_pre}, we use a unimodularity result to reduce the problem to an estimate on effective resistances, stated in Lemma~\ref{lemma_r_eff_Gil_Gab}. In Section~\ref{ss_comparison_Bernoulli}, we review some properties of the supercritical bond percolation graph on~$\ZZ^d$ and a stochastic domination result that will be crucial in proving Lemma~\ref{lemma_r_eff_Gil_Gab}. We then use this domination relation to prove the lemma for the supercritical random geometric graph in Section~\ref{ss_dom_rgg} and for the Gabriel graph in Section~\ref{ss_dom_Gab}. By subgraph comparison, this also holds for the Delaunay graph.

\subsubsection{Finite collision property}\label{ss_FCP_dim3_pre}
Recall the definition of the Green function~$\mathsf G$ from Section~\ref{ss_rw_and_en}, as well as its relation to the effective resistance in~\eqref{eq_Green_and_Resistance}. The following sections are devoted to proving the following result.

\begin{lemma}\label{lemma_r_eff_Gil_Gab}
For~$G \in \{\mathcal G_\mathrm{SGil},\mathcal G_\mathrm{Gab} \}$ we have,
\begin{equation}\label{eq_R_eff_1+delta}
\EE_0 [ \mathcal R_\mathrm{eff}(0 \leftrightarrow \infty;G )^{1+\delta} \cdot \mathds{1}\{ 0 \in V(G) \} ] < \infty, \quad \text{for some } \delta > 0.
\end{equation}
\end{lemma}
We prove this lemma for~$\mathcal G_\mathrm{SGil}$ in Section~\ref{ss_dom_rgg} and for~$\mathcal G_\mathrm{Gab}$ in Section~\ref{ss_dom_Gab}. Assuming it holds for the moment, we then use the unimodular results to complete the proof of the second part of the main theorem.

\begin{proof}[Proof of part two of Theorem~\ref{thm_main_dim}]
% We begin by considering the Voronoi skeleton graph~$\mathcal G_\mathrm{Vor}$. Since this graph is regular, the reversibility argument of the random walk in~\eqref{eq_pre_apply_MTP} implies that transience yields the finite collision property. The transience of~$\mathcal G_{\mathrm{Vor}}$ is established in part~(2) of Theorem~1 of~\cite{recurrenceVoronoiPPP}. 

Fix~$G \in \{\mathcal G_\mathrm{SGil},\mathcal G_\mathrm{Del},\mathcal{G}_\mathrm{Gab}$\}. We apply Theorem~\ref{thm_unimodularity_FCP}. We proceed analogously as the proof of the infinite collision property in Section~\ref{ss_prove_ICP_geo}. Define the~$\mathcal G_\bullet$-Borel set,
\begin{align*}
\mathcal A_\mathrm{FCP} = \{(G,\rho): \eqref{def_fin_coll_prop_disc} \text{ and } \eqref{def_fin_coll_prop_cont} \text{ holds for the walkers started from } \rho \}.
\end{align*}
Then, a similar application of the transfer lemma shows that it suffices to prove that~$\mathbf E[ \deg(\rho) \cdot \mathsf G(\rho) ] < \infty$. By the Mecke equation and the isometries of~$G$, this expected value is finite if
\begin{equation}\label{assumption_deg_G_0_2}
\EE_0 [ \deg_G(0) \cdot \mathsf G_G(0,0) \cdot \mathds{1}\{ 0 \in V(G) \} ] \stackrel{\eqref{eq_Green_and_Resistance}} = \EE_0 [ \deg_G^2(0) \cdot \mathcal R_\mathrm{eff}(0 \leftrightarrow \infty;G ) \cdot \mathds{1}\{ 0 \in V(G) \} ] < \infty.
\end{equation} 
We use Hölder’s inequality together with Lemma~\ref{lemma_fin_deg_to_unim} to deduce that~\eqref{eq_R_eff_1+delta} implies~\eqref{assumption_deg_G_0_2} for~$G \in \{\mathcal G_\mathrm{SGil},\mathcal{G}_\mathrm{Gab}\}$. Since the Gabriel graph~$\mathcal G_\mathrm{Gab}$ is a subgraph of the Delaunay triangulation~$\mathcal G_\mathrm{Del}$, we have
$$ \mathcal R_\mathrm{eff}(0 \leftrightarrow \infty;\mathcal G_\mathrm{Del}) \leq \mathcal R_\mathrm{eff}(0 \leftrightarrow \infty;\mathcal G_\mathrm{Gab}).$$
Therefore, Lemma~\ref{lemma_r_eff_Gil_Gab} suffices to conclude.
\end{proof}

\subsubsection{Comparison with Bernoulli bond percolation}\label{ss_comparison_Bernoulli}
The proof of Lemma~\ref{lemma_r_eff_Gil_Gab} relies on a coupling with Bernoulli bond percolation and on properties of its infinite connected component. We now revisit this model and collect the results needed in what follows. 
Given a graph~$G=(V,E)$ and a finite subset $K\subset V$, we define the (exterior) vertex and edge boundaries
\[
\partial_{V}K = \{y\in V \setminus K : y\sim x \text{ with }x\in K\},\quad \partial_{E}K=\{ \{x,y\} \in E: x\in K, y\in V \setminus K\} 
\] 

\noindent \textbf{Bond percolation on~$\ZZ^d$:} Let $\mathbf{P}_p$ denote the law of Bernoulli bond percolation on~$\mathbb{Z}^d$ with
parameter~$p\in[0,1]$ (we refer to \cite{grimmett1999percolation} for the definition of the model). Write~$p_c=p_c(d)$ for the corresponding critical threshold. Given~$p>p_c$, let~$\mathcal{C}_\infty$ be the infinite connected component, which exists and is unique~$\mathbf{P}_p$-a.s.  We denote by~$q_t(x,y)$ the heat kernel of the continuous time simple random walk on~$\mathcal{C}_\infty$, that is,
\begin{equation}\label{eq_q_t_heat_kernel_def}
q_t(x,y):= p^\mathrm{cont}_t(x,y) \cdot \deg_{\mathcal C_\infty}^{-1}(y), \qquad x,y\in \mathcal{C}_\infty(\omega).
\end{equation}
An important tool for our results is the following consequence of Theorem~1.1 and equation~(0.5) of~\cite{barlow2004random}.
\begin{proposition}\label{thm:bounds_HK_percolation}
Given $p>p_c$, there exists an event~$\Omega_1 \subseteq \Omega$ with~$\mathbf P_p(\Omega_1) = 1$, a universal constant~$c_1 = c_1(d,p)$ and a family of random variables~$\{S_x:x\in \mathbb{Z}^d\}$, with the property that~$S_x(\omega) < \infty$ when~$\omega \in \Omega_1$ and~$x \in \mathcal C_\infty(\omega)$, such that for all~$x \in \mathcal C_\infty$,
$$ q_t(x,x) \leq c_1 \cdot t^{-d/2} \quad \text{for } t \ge S_x(\omega). $$
Moreover, there exists $\nu = \nu(d,p)  > 0 $ such that for each~$x \in \ZZ^d$,
$$\mathbf{P}_p(x\in \mathcal{C}_{\infty},S_x\ge n)<\exp(-n^{\nu}) \quad \text{ for } n \text{ large enough.}$$
\end{proposition} 

% {\color{red} \begin{remark}\label{rem:upper_bound_greens_bernoulli_percolation}
%     Given that $0\in \mathcal C^{\infty}$, we have that
%     \[
%     g(0,0) \le \int_{n\ge 0} p_{t}(x,x) dt \le S_0+\int_{n\ge 0}^{S_0} c_1 t^{-d/2} dt \le S_0+C, 
%     \]
%     where $C$ is a constant that depends only on $d,c_1$.
% \end{remark}}

Consider the graph obtained by removing all vertices belonging to~$\mathcal C_\infty$. A natural question is whether the remaining graph contains an infinite connected component (consisting of both open and closed edges). Let $\mathcal{V}_{0}$ denote the vertices of the connected component containing the origin, with the convention that $\mathcal{V}_0=\emptyset$ if $0\in \mathcal{C}_{\infty}$. It was shown in~\cite[Theorem~4.2]{grimmett2014percolation} that, for~$p$ large enough, all such connected components are finite. In particular,~$\mathcal V_0$ is finite almost surely. The following result provides a quantitative version of this statement. 

\begin{lemma}\label{lem:size_cluster_origin_complment_infinitecomponent}
Let $p>p'$ for some $p'$ close enough to 1. There exists a dimension-dependent constant~$a>0$ such that,
    \begin{equation}
        \mathbf{P}_p (|\mathcal{V}_0|\ge M) \le \exp(- a \cdot M^{(d-1)/d} ) \quad \text{for } M \text{ large enough.} 
    \end{equation}
\end{lemma}
\begin{proof}
Let $S$ be a finite connected subset of~$\mathbb Z^d$ containing $0$, and suppose
that~$\mathcal V_0=S$. Every edge in~$\partial_{E} S$ is closed. Indeed, the endpoint outside~$S$ belongs to~$\mathcal C_\infty$, and if such an edge were open, then the endpoint inside $S$ would also be connected to~$\mathcal C_\infty$, contradicting~$S\subseteq \mathbb Z^d\setminus \mathcal C_\infty$. Thus, on the event $\{\mathcal V_0=S\}$, the set~$\partial_{E}S$
contains a closed cutset separating $0$ from infinity. Moreover, by the
edge-isoperimetric inequality on $\mathbb Z^d$,
$|\partial_{E} S|
    \ge c_d |S|^{(d-1)/d}$, where $c_d$ is a constant that depends on the dimension only. Consequently, on the event $\mathcal V_0$ is finite, 
\[
    \{|\mathcal V_0|\ge M\}
    \subseteq
    \left\{
    \begin{array}{c}
    \text{there exists a closed cutset surrounding }0\\
    \text{of cardinality at least } c_d M^{(d-1)/d}
    \end{array}
    \right\}.
\]
By a standard Peierls contour estimate (see \cite[section 1.4]{grimmett1999percolation}, for example), for $p$ sufficiently close to $1$,
there exists a constant $c \in(0,\infty)$ such that for~$n$ large enough,
\begin{equation}
    \mathbf P_p\big(
     \text{there exists a closed cutset surrounding }0
    \text{ of cardinality at least } n
    \big)
    \le e^{-c n}.
\end{equation}
The proof is finished by applying this estimate with $n=c_d M^{(d-1)/d}$ and setting~$a = c \cdot c_d$.

\end{proof}
By definition of~$\mathcal V_0$, we have~$\partial_V\mathcal V_0 \subset \mathcal C_\infty$. Combining the previous two results, we obtain the following statement, which will help us control the effective resistance to infinity for random elements of~$ \partial_V \mathcal V_0$.

\begin{lemma}\label{lemma:sup_reff_C_infty}
Let~$p>p'$ for some~$p'$ close enough to 1 and~$\nu$ be the constant of Proposition~\ref{thm:bounds_HK_percolation}. Then, there exists a dimensional-dependent constant~$b>0$ such that
$$ \mathbf P_p \left( \sup_{ \mathbf z \in \partial_V\mathcal V_0} \mathcal R_\mathrm{eff}( \mathbf z \longleftrightarrow \infty ; \mathcal C_\infty ) \ge u \right) \leq  \exp( -b u^\nu ), \quad \text{for } u \text{ large enough.} $$
\end{lemma}
\begin{proof}
Consider the following two facts. First, by \eqref{eq_Green_and_Resistance} and~\eqref{eq_q_t_heat_kernel_def}, we have that 
\begin{equation}\label{eq_R_eff_C_infty}
\mathcal R_\mathrm{eff}( \mathbf{z} \longleftrightarrow \infty; \mathcal C_{\infty} ) = \int_0^\infty q_t(\mathbf z,\mathbf z) \; \mathrm{d}t \leq S_\mathbf{z} + \int_1^{\infty} t^{-d/2}dt,
\end{equation}
where the inequality follows from bounding~$q_t$ by 1 and using the upper bound of Proposition~\ref{thm:bounds_HK_percolation}, and the fact that in our setting the continuous and discrete Green's functions coincide. The integral on the right-hand side of the last inequality is finite; we denote its value by~$A$. Second, since~$\mathcal V_0$ is a connected subset of~$\ZZ^d$, we have that~$\mathcal V_0 \subseteq [-|\mathcal V_0|, |\mathcal V_0|]^d$. Then, by Lemma~\ref{lem:size_cluster_origin_complment_infinitecomponent},
\begin{equation}\label{ineq_R_eff_sup}
\mathbf P_p \left( \sup_{\mathbf z \in \partial_V\mathcal V_0} \mathcal R_\mathrm{eff}( \mathbf z \longleftrightarrow \infty ; \mathcal C_\infty ) \ge u + A \right) \stackrel{\eqref{eq_R_eff_C_infty}} \leq \mathbf P_p \left( \sup_{\mathbf z \in \partial_V\mathcal V_0} S_\mathbf{z} \ge u, \mathcal V_0 \subseteq \Lambda_{n-1} \right) + \exp(-an^{(d-1)/d}),
\end{equation}
where~$\Lambda_n$ is the box of side-length~$n$ centered at the origin. Note that if~$\mathcal V_0 \subseteq \Lambda_{n-1}$, then~$\partial_V\mathcal V_0 \subseteq \Lambda_n$. By a union bound and Proposition~\ref{thm:bounds_HK_percolation}, for~$u$ large enough,
$$ \mathbf P_p \left( \sup_{z \in \partial_V\mathcal V_0} S_\mathbf{z} \ge u, \mathcal V_0 \subseteq \Lambda_{n-1} \right) \leq \sum_{\mathbf z \in \Lambda_n} \mathbf P_p( S_\mathbf{z} \ge u, \mathbf z \in \mathcal C_\infty ) \leq (2n+1)^d \cdot \exp(-u^\nu) $$
Taking~$n = \left \lceil u^{\frac{\nu d}{d-1}}\right \rceil$, we have~$\exp(-an^{(d-1)/d}) \leq \exp(-au^\nu)$. Plugging this and the last display into~\eqref{ineq_R_eff_sup} and setting~$b=\min\{1/2,a/2\}$ concludes the proof.
\end{proof}

\noindent \textbf{Comparison with Bernoulli percolation:} We now use the previous results for Bernoulli percolation to obtain estimates on the effective resistances of graphs generated by the Poisson point process~$\mathcal P$. To do so, we generalize a comparison scheme first introduced in \cite{recurrenceVoronoiPPP}. Let $G = \beta(\mathcal P)$ be a random graph. Given~$M>0$ and~$R \ge M/2$, we consider the set of boxes
\begin{equation}
    B_{\mathbf z} := M\mathbf{z}+[-R,R) ^d, \quad \mathbf z \in \mathbb Z^d.
\end{equation}
If~$R = M/2$, the boxes form a partition of~$\RR^d$, whereas if~$R>M/2$ the boxes overlap. The latter is not a problem and will actually be useful in some cases. On a suitable event that depends on~$G$, we select a random collection of \textit{good} boxes. We define
$$ X_e = \mathds{1}\{ B_\mathbf{x} \text{ and } B_\mathbf{y} \text{ are good} \}, \quad e = \{\mathbf x, \mathbf y  \} \in E(\ZZ^d). $$

We now state our comparison lemma.

\begin{lemma}\label{lem:comparison_effectiveresistance_percolation_domination}
    Suppose that the following is true:
    \begin{enumerate}
        \item On each good box $B_\mathbf z$ we can pick a vertex $v_{\mathbf z} \in V(G)$.
        \item There exists~$L > 0$ such that for any pair~$(\mathbf x, \mathbf y)$ satisfying~$X_{ \{\mathbf x, \mathbf y\} }=1$, the vertices~$v_{\mathbf x}$ and~$v_{\mathbf y}$ can be connected by a path $\gamma_{\mathbf x,\mathbf y}$ in $G$ totally contained within $B_{\mathbf x}\cup B_{\mathbf y}$ with length at most~$L$.
        \item The process~$\{X_e: e \in E(\ZZ^d)\}$ dominates a supercritical Bernoulli bond percolation process with parameter~$p$, which can be taken arbitrarily close to 1. 
    \end{enumerate}
    Then, there exists a probability space~$\widehat{\PP}$ on which the domination holds, where the following statement is satisfied. There exists a constant~$C = C(d,L,p,M,R)>0$ such that
    $$ \widehat{\PP} \left( \sup_{\mathbf z \in \partial_V\mathcal V_0} \mathcal R_\mathrm{eff}( \mathbf{z} \longleftrightarrow \infty ; \mathcal C_\infty ) \ge u \right) \leq \exp( -b u^\nu ), \quad \text{for } u \text{ large enough}. $$
    
    % Then, there exists $C=C(d,L,p)$ such that under the event $\mathbf 0 \in \mathcal C^{\infty}$, we have
    % \[
    %     \mathcal R(v_{\mathbf 0} \longleftrightarrow \infty , G) \le CL \mathcal R (\mathbf 0\longleftrightarrow \infty,\mathcal C^{\infty})
    % \]
\end{lemma}

\begin{proof}
Fix~$\mathbf z \in \mathcal C_\infty$ and denote by~$\theta$ an arbitrary unit flow from~$\mathbf z$ to infinity in~$\mathcal C_\infty$. We now use the hypothesis to construct a unit flow~$\Theta$ from~$v_\mathbf{z}$ to infinity in~$G$. Given~$\mathbf z_1,\mathbf z_2 \in \mathcal C_\infty$ with~$\mathbf z_1 \sim \mathbf z_2$ and~$x,y \in V(G)$ with~$\{x,y\} \in \gamma_{\mathbf z_1,\mathbf z_2}$, we define
    \[
        \overrightarrow{\mathbf z_1 \mathbf z_2} \big(\overrightarrow{xy} \big) := \begin{cases}
            \overrightarrow{\mathbf z_1 \mathbf z_2}, \quad \text{if $\gamma_{\mathbf z_1,\mathbf z_2}$ crosses $\{x,y\}$ from $x$ to $y$}\\[2mm]
            \overrightarrow{\mathbf z_2 \mathbf z_1}, \quad \text{if $\gamma_{\mathbf z_1,\mathbf z_2}$ crosses $\{x,y\}$ from $y$ to $x$},
        \end{cases}
    \]
which in turn defines the unit flow,
    \[
        \Theta(\overrightarrow{xy}) := \sum_{ \{\mathbf z_1,\mathbf z_2\} \in E(\mathcal C_\infty) } \theta( \overrightarrow{ \mathbf z_1 \mathbf z_2 }(\overrightarrow{xy}) ) \cdot \mathds{1}\{ \{x,y\} \in \gamma_{\mathbf z_1, \mathbf z_2} \}, \quad x,y \in V(G).
    \]
    It readily follows that this defines a unit flow in~$G$. Note that the flow through edges that do not belong to any path~$\gamma$ is zero. Also, an edge in $G$ may belong to several paths $\gamma$. The number of such paths is bounded by the number of paths that can pass through a given edge, which depends only on the dimension~$d$ and the size~$R$ of the boxes. We denote this number by~$C = C(d,R)$. Then, for any edge $\{x,y\} \in E(G)$,
    \[
        |\Theta(\overrightarrow{xy})| \le C \cdot |\theta( \overrightarrow{\mathbf z_1\mathbf z_2})| \quad \text{for each } \mathbf z_1, \mathbf z_2 \text{ with } \{x,y\} \in \gamma_{\mathbf z_1,\mathbf z_2},  
    \]
    Recall that $|\gamma_{\mathbf z,\mathbf z'}|\le L$. Then,
    \begin{equation*}
        2 \cdot \mathcal{R}_\mathrm{eff}( v_\mathbf{z} \leftrightarrow \infty;G) \leq \sum_{\{x,y\}\in E(G)} \Theta (\overrightarrow{xy})^2 \le C^2L \cdot \sum_{\{\mathbf z_1,\mathbf z_2\}\in E({\mathcal C_\infty})} \theta(\overrightarrow{\mathbf z_1 \mathbf z_2})^2.
    \end{equation*}
    Hence, by taking the infimum over all unit flows on the right-hand side, we obtain
    $$ \mathcal{R}_\mathrm{eff}( v_\mathbf{z} \leftrightarrow \infty;G) \leq C^2 L \cdot \mathcal{R}_\mathrm{eff}( \mathbf{z} \leftrightarrow \infty;\mathcal C_\infty).$$
Finally, the proof is completed by applying Lemma~\ref{lemma:sup_reff_C_infty}.
\end{proof}

\subsubsection{Domination of Supercritical Bond Percolation for the Random Geometric Graph}\label{ss_dom_rgg}
In this section, we prove Lemma~\ref{lemma_r_eff_Gil_Gab} for the supercritical Gilbert graph~$\mathcal G_\mathrm{SGil}$. The proof is divided into two parts. First, we present a scheme to verify the assumptions of Lemma~\ref{lem:comparison_effectiveresistance_percolation_domination}. Then, we use this lemma to estimate the effective resistance and derive~\eqref{eq_R_eff_1+delta}.\\

\noindent \textbf{Domination of Bernoulli percolation.} We denote by~$\mathbf{e}_k$ the~$k$-th canonical basis vector in~$\RR^d$ and define
$$ B_{\mathbf{z}} =\mathbf{z} M + [-M,M]^d, \quad \mathbf z \in \ZZ^d.$$
Note that with this choice of side length, for $\mathbf{z}\in \mathbb{Z}^d$, the intersection of $B_{\mathbf{z}}, B_{\mathbf{z}+\mathbf{e}_1}$ is a rectangle with smallest side length equal to $M$. We employ the renormalization scheme of \cite{FriedrichSauerwaldStauffer}. Let~$k \in \{1,\dots,d\}$. We say that the box~$B$ contains a~$k$-crossing cluster if there exists a connected component~$(V_B,E_B)$ of the random geometric graph within~$B$ such that its~$r$-fattening,
$$ V_B(r):= \bigcup_{z \in V_B } B_\mathrm{Eucl}(z,r) \subset \RR^d,$$
intersects both faces of~$B$ orthogonal to the~$k$-th coordinate axis. The diameter of this connected component is defined as the Euclidean diameter of~$V_B(r)$:
$$\sup\{\norm{x-y}: x,y \in V_B(r)\}.$$
For each~$\mathbf z \in \ZZ^d$, consider the event
\begin{align*}
    & \mathcal A_\mathbf{z} = \left\{ \begin{array}{c}
        B_\mathbf{z} \text{ has a unique cluster of diameter greater than or equal to } \frac{2M}{5} \\[1mm]
        \text{and this cluster is } k\text{-crossing for every } k \in \{1,\dots,d\}  
    \end{array} \right\}
\end{align*}
We refer to the cluster appearing in the definition of~$\mathcal A_{\mathbf z}$ as the \textit{giant cluster} of the box, or simply giant. We say that~$B_\mathbf{z}$ is a good box if the event~$\mathcal A_\mathbf{z}$ holds. We now declare the elements of the comparison lemma.
\begin{enumerate}
    \item For a good box $B_{\mathbf z}$, choose as the representative vertex $v_{\mathbf z}$ to be the closest vertex to the center of $B_{\mathbf z}$ that belongs to the giant cluster of the box which is unique a.s.
    \item Each vertex~$x$ in~$\mathcal G_\mathrm{SGil}$ is connected to all the vertices in the ball~$B_\mathrm{Eucl}(x,r)$, then there exists~$L = L(d,M)$ such that each self-avoiding path in~$\mathcal G_\mathrm{Gil}$ inside two contiguous boxes has length at most~$L$. If~$B_{\mathbf{z}_1}$ and~$B_{\mathbf{z}_2}$ are good, then their giant clusters are connected by a path in~$\mathcal G_\mathrm{SGil}$. Thus~$v_{\mathbf{z}_1}$ and~$v_{\mathbf z_2}$ are connected within~$B_{\mathbf{z}_1} \cup B_{\mathbf{z}_2}$ through the giant cluster.
    \item By \cite[Lemma 11]{FriedrichSauerwaldStauffer}, we have that~$\mathbb P(\mathcal A_{\mathbf{z}})\to 1$ as $M\to \infty$. The state of a bond only depends on contiguous boxes, and the probability that a bond is open can be arbitrarily made close to~$1$ by making~$M$ large. By \cite[Theorem~0.0]{liggettdomination}, the latter bond percolation process dominates a supercritical Bernoulli bond percolation on~$\mathbb Z^d$. Note that in the definition of the coupling the origin was not taken into account\footnote{The main reason for this is that goodness is not an increasing event.}.
\end{enumerate}
Therefore, all the assumptions of Lemma~\ref{lem:comparison_effectiveresistance_percolation_domination} are satisfied. \\

\noindent \textbf{Electric resistance.} We use the stochastic domination assumption for some fixed~$M$ with~$p$ sufficiently close to~$1$ for which the assumptions of Lemma~\ref{lem:comparison_effectiveresistance_percolation_domination} are satisfied. Recall the definition of $\mathcal V_0$ and~$\partial_V\mathcal V_0$ from Section~\ref{ss_comparison_Bernoulli}. We define
$$ A_{\mathbf{z}}=\mathbf{z} M + [-M/2,M/2)^d, \quad \mathbf z \in \ZZ^d.$$
Note that the boxes $A_{\mathbf z}$ form a partition of~$\RR^d$. We now use this partition to show that there exists a random vertex~$\mathbf z^\ast \in \partial_V\mathcal V_0$ such that the origin and~$v_{\mathbf z^\ast}$ are connected by a path in~$\mathcal G_\mathrm{SGil}$ contained in
$$\bigcup \{B_\mathbf{z}: \mathbf{z} \in \mathcal V_0 \cup \partial_V\mathcal V_0 \}.$$

We first analyze the case $\mathcal{V}_0\ne \emptyset$. Consider the shortest path $\gamma$ of $\mathcal G_\mathrm{SGil}$ between the origin and $\cup_{\mathbf z \in\partial_V \mathcal V_0} A_{\mathbf z}$, ending at the first vertex such that the edge connecting it to the previous vertex crosses the set $\cup_{\mathbf z \in\partial_V \mathcal V_0} A_{\mathbf z}$ (note that the last vertex may not be included in the latter set). This path exists a.s. given that $\mathbf 0 \in V(G_\mathrm{SGil})$ and~$\mathcal G_\mathrm{SGil}$ is an infinite graph. Denote by $\mathbf{z}^*$ the site for which the box $A_{\mathbf z^*}$ is the first one of $$\{A_{\mathbf z}:\mathbf{z}\in \partial_V \mathcal V_0 \}$$ that is visited by the last edge of $\gamma$ (this is well defined by the disjointness of boxes $A_{\mathbf z}$ and the definition of $\mathcal V_0$). Denote by $x^*$ the last vertex of $\gamma$, which may not belong to $A_{\mathbf z^*}$, but is at distance at most $r$. We claim that $x^*$ belongs to the giant connected component of $B_{\mathbf{z}^*}$. Indeed, as $x^*$ is at distance at most $r$ from $ A_{\mathbf z^*}$ and the origin is the first vertex of $\gamma$, the path $\gamma$ must cross from $\partial B_{\mathbf z^*}$ to $\partial A_{\mathbf z^*}$ (in the worst case $\{\mathbf 0\} = \mathcal V_0$, and by the side length of the cubes $B$, the origin vertex belongs to some boundaries of cubes $B$). Denote this crossing by $\gamma'$. More precisely, $\gamma'$ is the sub-path of $\gamma$ whose initial vertex is the first vertex of $\gamma$ from which the path reaches $x^*$ before leaving $B_{\mathbf z^*}$ and last vertex equal to $x^*$. $\gamma'$ has Euclidean diameter at least $M/2-2r$, is totally contained within $B_{\mathbf z^*}$ and contains $x^*$. Then, as the box $B_{\mathbf z}^*$ is good, we have that $\gamma'$ belongs to the giant in $B_{\mathbf z^*}$, by definition of good box and that this connected component has diameter larger than $2M/5$ (we make $M$ large enough in this case). This finishes the proof of the claim. In the case $\mathcal V_0=\emptyset$, by definition, we have that $\mathbf{0} \in \mathcal C^{\infty}$. Then, we just take $x^*$ to be the origin. 

As a consequence of the triangle inequality and the trivial upper bound of the effective resistance, by~\eqref{ineq_R_eff_distance_G}, we have
$$ \mathcal R_\mathrm{eff}(0 \longleftrightarrow \infty; \mathcal G_\mathrm{SGil}) \le d_{\mathcal G_\mathrm{SGil}}(0,v_{\mathbf z^\ast}) + \mathcal R_\mathrm{eff}(v_{\mathbf z^\ast}\longleftrightarrow \infty; \mathcal G_\mathrm{SGil}). $$
We first bound the first term of the right hand side. As each box $B_\mathbf z$ has volume $M^d$, we can connect any two vertices (that can be connected) inside the box by a path made of at most $cM^d$ vertices as $r$ is a fixed constant. Then, since~$v_{\mathbf{z}^\ast} \in \partial_V \mathcal V_0 $, we have that~$d(0,v_{\mathbf z^\ast})\le 2cM^d|\mathcal V_0|$. For the second term we note that  
$$ \mathcal R_\mathrm{eff}(v_{\mathbf z^\ast}\longleftrightarrow \infty; \mathcal G_\mathrm{SGil}) \leq \sup_{\mathbf z \in \partial_V \mathcal V_0} \mathcal R_\mathrm{eff}(v_{\mathbf z}\longleftrightarrow \infty; \mathcal G_\mathrm{SGil}).$$
Therefore, applying together Lemma \ref{lem:size_cluster_origin_complment_infinitecomponent} and Lemma~\ref{lem:comparison_effectiveresistance_percolation_domination} we have Lemma~\ref{lemma_r_eff_Gil_Gab} for~$\mathcal G_\mathrm{SGil}$.

\subsubsection{Domination by Supercritical Bond Percolation for the Gabriel graph}\label{ss_dom_Gab}
In this section, we prove Lemma~\ref{lemma_r_eff_Gil_Gab} for the Gabriel graph~$\mathcal G_\mathrm{Gab}$. As we did for~$\mathcal G_\mathrm{SGil}$ we divide the proof into two sections. \\

\noindent\textbf{Domination with Bernoulli percolation.} We use the renormalization scheme for the Gabriel graph from~\cite{recurrenceVoronoiPPP}. For~$M \ge 1$, we consider the boxes
$$ B_\mathbf{z}:= M \mathbf{z} + [-M/2,M/2)^d, \quad \mathbf{z} \in \ZZ^d. $$
These boxes partition~$\RR^d$. Given~$m \in \NN$, we consider an integer number~$\alpha = \alpha(d,m)$ and declare~$B_\mathbf{z}$ a good box if when~$B_\mathbf{z}$ is cut into~$\alpha^d$ sub-boxes of side-length~$M/\alpha$, each of these sub-boxes contains at least one and at most~$m$ points of~$\mathcal P \cup \{0\}$. 

We now declare the elements of the comparison lemma. For each good box~$B_\mathbf{z}$, we choose the representative vertex~$v_\mathbf{z} \in V(\mathcal G_\mathrm{Gab})$ as the closest vertex to~$M\mathbf z$ within the sub-box of side-length~$M/\alpha$ centered at~$M\mathbf{z}$. For assumption 2 and 3 we refer Section~5.2 of~\cite{recurrenceVoronoiPPP}. Moreover, using the same computations, we can choose the integers~$\alpha$ and~$M$ so that~$\{ X_e: e \in E(\ZZ^d) \}$ dominates a supercritical bond percolation on~$\ZZ^d$ with parameter~$p$, which can be taken arbitrarily close to~$1$. This choice of parameters also yields the following result, which will be useful to obtain a random vertex~$z^\ast \in \mathcal V^\mathrm{ext}_0$ as we did before. This result is established in Section~5.2.3 of~\cite{recurrenceVoronoiPPP}. We refer to that section for the proof. 

\begin{lemma}\label{lemma_sub_boxes_Gab}
Given the nearest neighbors~$\mathbf x, \mathbf y \in \ZZ^d$ with~$B_\mathbf{x},B_\mathbf{y}$ good boxes, let~$Q_1$ and~$Q_2$ be two contiguous sub-boxes of side-length~$M/\alpha$ in the chain of sub-boxes intersecting the line segment~$[M\mathbf{x}, M\mathbf{y}]$. For any~$x,y \in V(\mathcal G_\mathrm{Gab})$ with~$x \in Q_1$ and~$y \in Q_2$, there exists a path in~$\mathcal G_{\mathrm{Gab}}$ from~$x$ to~$y$ contained in~$B(z_1, M/2) \cup B(z_2, M/2)$, where~$z_1$ and~$z_2$ denote the centers of~$Q_1$ and~$Q_2$, respectively.
\end{lemma}

The previous lemma extends from line segments to~$(d-1)$-dimensional cubes. The proof is identical, since the argument depends only on geometric constraints (Lemma~14) and on being surrounded by good boxes (Lemma~15), both of which hold in this setting.

\begin{corollary}\label{cor_Gab_conn}
Fix~$\mathbf z \in \ZZ^d$ and~$i \in \{1,\dots,d\}$, and define
\begin{align*}
& \mathcal B(\mathbf z, i) := \{ B_{\mathbf z + \mathbf u}: \mathbf u \in \{0,1\}^d, \mathbf u \cdot \mathbf{e}_i = 0\}, \\[2mm]
& C(\mathbf z,i) := \mathrm{conv}( \{M(\mathbf z + \mathbf u): \mathbf u \in \{0,1\}^d, \mathbf u \cdot \mathbf e_i = 0 \} ), \\[2mm]
& \mathcal Q(\mathbf z,i) := \{ Q \text{ sub-box of side-length } M/\alpha: Q \text{ is contiguous to } Q', \; Q' \cap C(\mathbf z,i) \neq \varnothing  \}.
\end{align*}
Assume the boxes in~$\mathcal B(\mathbf z, i)$ are good. Let~$Q_1,Q_2 \in \mathcal Q(\mathbf z,i)$ be contiguous sub-boxes with~$Q_2 \cap C(\mathbf z,i) \neq \varnothing$, and~$x \in V(\mathcal G_\mathrm{Gab}) \cap Q_1$ and~$y \in V(\mathcal G_\mathrm{Gab}) \cap Q_2$. Then, there exists a path in~$\mathcal G_\mathrm{Gab}$ from~$x$ to~$y$ contained in~$\cup \{B: B \in \mathcal B(\mathbf z, i) \}$.    
\end{corollary}
The family of sub-boxes~$\mathcal Q(\mathbf z,i)$ is the natural analogue of the sub-boxes intersecting the line segment, together with their neighbours. Note that we do not require~$Q_1 \cap C(\mathbf z,i) \neq \varnothing$. We only require~$Q_1$ to be contiguous to~$Q_2$.  \\

% It guarantees the existence of a probability space~$\mathbf{P}$ supporting the point process~$\mathcal P \cup \{\mathbf{0}\}$ with the process~$\{X_e: e \in E(\ZZ^d) \}$ and an i.i.d. Bernoulli field~$\{\sigma_e : e \in E(\ZZ^d)\}$ with parameter~$p > p_c$ such that
% $$ \sigma_e \leq X_e \quad \text{for each } e \in E(\ZZ^d).$$
% We write~$\mathcal C_\infty$ for the infinite cluster of~$\sigma$. 

\noindent \textbf{Electric resistance.} We use the stochastic domination assumption for some fixed~$M$ and~$m$ with~$p$ sufficiently close to~$1$ for which the assumptions of Lemma~\ref{lem:comparison_effectiveresistance_percolation_domination} are satisfied. We now use Corollary~\ref{cor_Gab_conn} to show that there exists a random vertex~$\mathbf z^\ast \in \mathcal V^\mathrm{ext}_0$ such that the origin and~$v_{\mathbf z^\ast}$ are connected by a path in~$\mathcal G_\mathrm{Gab}$ contained in 
$$ \bigcup \{ B_\mathbf{z}: \mathbf z \in \mathcal V_0 \cup \mathcal{V}^\mathrm{ext}_0 \}.$$

Recall that~$ \mathcal{V}^\mathrm{ext}_0 \subset \mathcal C_\infty$. Then, the boxes 
$$ \{ B_{\mathbf z}: \mathbf z \in \partial^\mathrm{ext} \mathcal V_0 \} $$
are good and separate the origin from infinity. Inside this family of boxes, we construct a closed~$(d-1)$-dimensional surface made of convex hulls of the centers of these boxes,
$$ S := \bigcup \left\{ \mathrm{conv}( M \mathbf z_1,\dots,M \mathbf z_{2^{d-1}} ): \mathbf z_i \in \mathcal V^\mathrm{ext}_0, \; \mathrm{conv}( M \mathbf z_1,\dots,M \mathbf z_{2^{d-1}} ) \text{ is a } (d-1)\text{-cube}  \right\}. $$
Note that~$\mathrm{conv}( M \mathbf z_1,\dots,M \mathbf z_{2^{d-1}} )$ is a~$(d-1)$-cube if it is orthogonal to~$\mathbf e_i$ for some~$i \in \{1,\dots,d\}$. Let~$(u_0,u_1,u_2,\dots)$ be a path in~$\mathcal G_\mathrm{Gab}$ connecting~0 with infinity. By connectivity of the graph and the definition of the closed surface~$S$, there exists~$j \in \NN$ such that
\begin{equation}\label{eq_crossing_Gab}
\begin{array}{ll}
    & [u_j,u_{j+1}] \cap C(\mathbf z,i) \neq \varnothing \text{ and} \\[2mm]
    & \text{the boxes in } \mathcal B(\mathbf z,i) \text{ are good}
\end{array}
\qquad \text{for some } i \in \{1,\dots,d\} \text{ and } \mathbf{z} \in \mathcal C_\infty.
\end{equation}
The connectivity property of the Gabriel graph ensures that the sub-box containing~$u_j$ or that containing~$u_{j+1}$ belongs to~$\mathcal Q(\mathbf z,i)$. Consequently, by Corollary~\ref{cor_Gab_conn}, there exists a path~$\gamma_\mathbf{z}$ in~$\mathcal G_\mathrm{Gab}$ from~$u_j$ to~$v_\mathbf{z}$ contained in~$B_{\mathbf z_1} \cup \cdots \cup B_{\mathbf{z}_{2^{d-1}}}$, where~$v_\mathbf{z}$ is the associated vertex of the good box~$B_\mathbf{z} \in \mathcal B(\mathbf z,i)$. 

Let~$j$ be the minimal index satisfying~\eqref{eq_crossing_Gab}. Then there exists a (random) vector~$\mathbf z \in \mathcal V^\mathrm{ext}_0$ and a path in~$\mathcal G_\mathrm{Gab}$ from the origin to~$v_\mathbf{z}$, given by~$\{u_0=0,\dots,u_j\} \cup \gamma_\mathrm{z}$, such that~$v_\mathbf{z}$ satisfies the desired properties. The properties of~$\mathbf z$ allow us to argue exactly as in the case of~$\mathcal G_\mathrm{SGil}$. Applying Lemma~\ref{lem:size_cluster_origin_complment_infinitecomponent} and Lemma~\ref{lem:comparison_effectiveresistance_percolation_domination}, we deduce~\eqref{lemma_r_eff_Gil_Gab} for~$\mathcal G_\mathrm{Gab}$.

\subsection{Trivial tail~$\sigma$-algebra}\label{ss_trivial_tail_G}\label{ss_dim_2_trivial}
In this section, we prove that the tail~$\sigma$-algebra of a simple random walk on~$\mathcal G_\mathrm{SGil}, \mathcal G_\mathrm{Del}$ and~$\mathcal G_\mathrm{Gab}$ is trivial almost surely. Thus completing the proof of Theorem~\ref{thm_main_dim}. 

\begin{proof}[Proof of part 3 of Theorem~\ref{thm_main_dim}]
Let~$G \in \{ \mathcal G_\mathrm{SGil}, \mathcal G_\mathrm{Del}, \mathcal G_\mathrm{Gab} \}$. Using the unimodular measure~\eqref{eq_unimodular_H}, we see that the expected value in~\eqref{limit_log_ball} equals
\begin{equation}\label{eq_trivial_G_1}
\begin{aligned}
\gamma_G^{-1} \cdot \EE\Big[ \sum_{x \in \mathcal P}  \deg_G(x) & \cdot  \log |B_G(x,n)|  \cdot \mathds{1} \{x \in V(G) \cap [0,1]^d \} \Big] \\[2mm]
& = \gamma_G^{-1} \cdot \EE_0[ \deg_G(0) \cdot \log |B_G(0,n)| \cdot \mathds{1} \{0 \in V(G) \} ],
\end{aligned}
\end{equation}
where the equality follows from applying the Mecke equation together with the isometries of~$G$. Note that by translation invariance~$\PP_0(0 \in V(G))>0$. By Hölder's inequality,
\begin{equation}\label{eq_trivial_G_2}
\begin{aligned}
\EE_0[ \deg_G(0) & \cdot \log|B_G(0,n)|  \cdot  \mathds{1} \{0 \in V(G) \} ] \\[2mm]
& \leq \EE_0[ \deg_G(0)^p \cdot \mathds{1}\{0 \in V(G)\}]^{1/p} \cdot \EE_0[ (\log |B_G(0,n)|)^q \cdot \mathds{1} \{0 \in V(G) \} ]^{1/q} \\[2mm]
& \leq \EE_0[ \deg_G(0)^p \cdot \mathds{1}\{0 \in V(G)\}]^{1/p} \cdot  \log  \EE_0[ |B_G(0,n)| \mid 0 \in V(G) ],
\end{aligned}
\end{equation}
where the second inequality holds for all sufficiently large~$n$, since~$x \mapsto (\log x)^q$ is eventually concave and~$|B_G(0,n)| \ge n+1$, which allows us to apply Jensen’s inequality. Since we have good control on the expected degree of the root (see Lemma~\ref{lemma_fin_deg_to_unim}), the only remaining step in applying Proposition~\ref{thm_Liouville} is to establish
\begin{equation}\label{limit_apply_trivial_RGG}
\lim_{n \to \infty} n^{-1} \cdot \log \EE_0 [|B_G(0,n)| \mid 0 \in V(G)] = 0.
\end{equation}
It then follows from Proposition~\ref{thm_Liouville} together with the transfer lemma, Lemma~\ref{lemma_translate_as}, that a simple random walk on~$G$ has trivial tail~$\sigma$-algebra almost surely.

We begin with~$\mathcal G_\mathrm{SGil}$. Note that on~$\PP_0$,
$$ B_{\mathcal G_\mathrm{SGil}}(0,n) \subseteq \mathcal P \cap B_\mathrm{Eucl}(0,rn) \quad \text{for each } n \in \NN. $$
Then, the expected value in~\eqref{limit_apply_trivial_RGG} is smaller than
$$\EE_0[ |\mathcal P \cap B_\mathrm{Eucl}(0,rn) \mid 0 \in V(\mathcal G_\mathrm{SGil}) ] \leq \gamma_{\mathcal G_\mathrm{SGil}}^{-1} \cdot \left( 1 + \EE | \mathcal P \cap B_\mathrm{Eucl}(0,rn) | \right). $$
For this graph~$\gamma_{\mathcal G_\mathrm{SGil}} = \PP_0( 0 \in V(\mathcal G_\mathrm{SGil}) )$. Finally, the Poissonian property shows that~\eqref{limit_apply_trivial_RGG} holds for this random graph. 

Since the Gabriel graph is a subgraph of~$\mathcal G_\mathrm{Del}$, it suffices to establish~\eqref{limit_apply_trivial_RGG} for the Delaunay graph; the same limit then holds for~$\mathcal G_\mathrm{Gab}$. For this, we use the following result. Recall the definition of Voronoi cell at the beginning of Section~\ref{ss_rgg_et_all}.

We call a \emph{Voronoi polyomino} a connected union of Voronoi cells, where connectedness is understood in the corresponding set of edges of the Delaunay graph. Given~$m \in \NN$, let~$\Pi_m$ denote the collection of all polyominoes~$V$ containing the origin such that~$|\mathcal P \cap V| \leq m-1$. For each~$\mathbf z \in \ZZ^d$ let~$B_\mathbf{z} := \mathbf z + [-1/2,1/2)^d$, and for each connected set~$C \subseteq \RR^d$ let
$$ \mathbf A(C) := \{ \mathbf z \in \ZZ^d : B_\mathbf{z} \cap C \neq \varnothing \}. $$
The following theorem is an adaptation of Theorem~1 in~\cite{Pimentel13}. We focus only on the upper bound inequality. In the original work, the estimates are proved for Voronoi polyominoes intersecting a fixed bounded region around the origin. Since the Palm version of the Poisson point process differs from the original process only by the insertion of a deterministic point at the origin, the same arguments and estimates remain valid under the Palm measure, up to a modification of constants.

\begin{theorem}
There exist constants~$a,b > 0$ such that if~$s \ge a m$, then 
$$ \PP_0 \left( \max_{V \in \Pi_m} |\mathbf A(V)| \ge s \right) \leq e^{-bs} $$
\end{theorem}

Now, fix~$n \in \NN$ and define~$X_n := \max_{V \in \Pi_n} |\mathbf A(V)|$. Applying the previous theorem, there exists~$C = C(a,b,d)>0$ such that
\begin{equation}\label{expectation_Palm_X_Del}
\EE_0 X^d_n \leq C n^d.
\end{equation}
We claim that
\begin{equation}\label{last_inclusion}
B_{\mathcal G_\mathrm{Del}}(0,n) \subseteq \mathcal P \cap [-X_n,X_n]^d.
\end{equation}
In fact, let~$x \in B_{\mathcal G_\mathrm{Del}}(0,n)$, then there exists a path~$\{x_0 = 0,x_1,\dots,x_n = x\}$ in~$\mathcal G_\mathrm{Del}$ such that~$V_x := \cup_i \mathrm{Vor}_\mathcal{P}(x_i)$ is a Voronoi polyomino in~$\Pi_n$. Then, 
$$ x \in \mathcal P \cap \bigcup_{\mathbf z \in \mathbf{A}(V_x)} B_\mathbf{z} \subseteq \mathcal P \cap \bigcup_{\mathbf z \in [-|\mathbf A(V_x)|, |\mathbf A(V_x)|]^d} B_\mathbf{z}, $$
where the last inclusion follows from the fact that a connected subset of~$\ZZ^d$ with~$m$ vertices and containing the origin must be contained in~$[-m,m]^d$. Since~$V_x \in \Pi_n$, we have that~$|\mathbf A(V_x)| \leq X_n$ and then the inclusion~\eqref{last_inclusion} is proved. Using this inclusion, we have
$$ \EE_0 B_{\mathcal G_\mathrm{Del}}(0,n) \leq \EE_0 (2X_n+1)^d, $$
together with~\eqref{expectation_Palm_X_Del} we have established~\eqref{limit_apply_trivial_RGG}. This completes the proof.
\end{proof}

\section{Long-range percolation in the Euclidean lattice}\label{ss_long_range_perc}
In this section, we apply the results on unimodular graphs from Section~\ref{ss_unimodularity} to study the collision properties of random walks on long-range percolation clusters. Since the results for the long-range percolation model that we will use are unaffected by the choice of norm on~$\RR^d$ throughout this section we continue to denote by~$\norm{ \cdot }$ the Euclidean norm.

This random graph is constructed as follows. Given~$\beta > 0$ and~$s > 0$, let~$\mathsf P_{\beta,s}:\RR \to [0,1)$ be a function such that
\begin{equation}\label{bond_prob_long_range}
\lim_{r \to \infty} \frac{\mathsf P_{\beta,s}(r)}{ \beta r^{-s} } = 1
\end{equation}
The long-range percolation graph is the random graph with vertex set~$\ZZ^d$, where each edge~$\{x,y\}$, with $x,y \in \ZZ^d$, is present independently with probability~$p_{x,y} = \mathsf P_{\beta,s}(\norm{x-y})$. 

Note that~$\sum_{x \in \ZZ^d}\mathsf P_{\beta,s}(\norm{x}) < \infty $ if and only if~$s > d$. In this case, the graph is locally finite almost surely, and random walks can be defined on it.\\

\noindent \textbf{Infinite connected component.} By ergodicity of the model, the translation-invariant event of the existence of an infinite cluster has probability either zero or one. We say that~$\mathsf P_{\beta,s}$ is~\emph{percolating} if the model admits an infinite connected component. It is proved in~\cite{AizenmanKestenNewman87} that, in this case, the infinite connected component is almost surely unique.

We collect below several facts about long-range percolation with edge probabilities~$\mathsf P_{\beta,s}$ satisfying~\eqref{bond_prob_long_range}. They are taken from~\cite[Theorem 1.2]{Berger02}, we refer the reader there for further details. 
\begin{itemize}
    \item For every~$d \ge 1$ with~$s>d$ and~$\beta>0$, there exists a non-percolating function~$\mathsf P_{\beta,s}$. 
    \item For~$d=1$ with~$s \in (1,2)$, there exists~$\beta > 0$ sufficiently large such that there exists a percolating function~$\mathsf P_{\beta,s}$. 
    \item For~$d=1$ with~$s=2$ and~$\beta>1$, there exists a percolating function~$\mathsf P_{\beta,s}$, whereas for \(\beta \le 1\), there exists a non-percolating function \(\mathsf P_{\beta,s}\).
    \item For~$d \ge 2$ with~$s>d$, there exists~$\beta > 0$ sufficiently large such that there exists a percolating function~$\mathsf P_{\beta,s}$.
\end{itemize}
We denote by~$\mathcal C_\infty$ the unique infinite connected component.\\

\noindent \textbf{Unimodularity.} Consider the graph~$(\ZZ^d,E_\mathrm{all}(\ZZ^d))$, where
$$ E_\mathrm{all}(\ZZ^d) = \{ \{x,y\}: x,y \in \ZZ^d \text{ and } x \neq y \}.$$
This graph, rooted at the origin, is a deterministic unimodular graph. An application of~\cite[Example~9.4]{AldousLyons2007} shows that~$(\mathcal C_\infty,0)$, conditioned on the root belonging to~$\mathcal C_\infty$, is a unimodular random rooted graph.

By translation invariance (of the law of the long-range percolation model), we have that~$\PP(0 \in \mathcal C_\infty) > 0$. Then, the analogue of the transfer lemma (Lemma~\ref{lemma_translate_as}) holds: for any translation invariant event~$A$, ergodicity yields
\begin{equation}\label{transfer_lemma_long_range}
\PP(A \mid 0 \in \mathcal C_\infty) = 1 \quad \Longrightarrow \quad \PP(A) = 1.
\end{equation}
The other direction is clear to hold for any event. 

We also have good control on the degree of the root. For every~$d \ge 1$ with~$s > d$,
\begin{equation}\label{root_long_range_percolation}
\EE[ \exp(\theta \cdot \deg(0) ) ] \leq \exp \left( (e^\theta-1)\beta \cdot \sum_{x \in \ZZ^d} \norm{x}^{-s} \right) < \infty \quad \text{for all } \theta \in \RR.
\end{equation}

We are ready to start the proof of Theorem~\ref{thm:long_range_cluster}. We begin with the triviality of the tail~$\sigma$-algebra of the random walk. 
\begin{proof}[Proof of part three of Theorem~\ref{thm:long_range_cluster}]
The triviality of the tail~$\sigma$-algebra is established in Section~5.2 of~\cite{BenjaminiCurien2012}. The proof extends to all~$d \ge 1$ with~$s \in (d,2d)$, as it relies on a chemical distance result from~\cite[Theorem~3.1]{Biskup11}, which holds in this regime. Hence, on~$\{0 \in \mathcal C_\infty\}$, the~$\sigma$-algebra~$\mathcal T_\mathrm{rw}(\mathcal C_\infty)$ is trivial. Since the event~$\{ \mathcal T_\mathrm{rw}( \mathcal C_\infty) \text{ is trivial} \}$ is translation invariant, by~\eqref{transfer_lemma_long_range}, we have that~$\mathcal T_\mathrm{rw}(\mathcal C_\infty)$ is trivial almost surely.

\end{proof}

\subsection{Regime $d=1$ with $s = 2$ and $d=2$ with $s \ge 4$}\label{ss_reg_long_rec}
In this regime, we prove that~$\mathcal C_\infty$ has the infinite collision property almost surely whenever the edge probabilities given by a percolating function~$\mathsf P_{\beta,s}$ that satisfies~\eqref{bond_prob_long_range}.

\begin{proof}[Proof of part one of Theorem~\ref{thm:long_range_cluster}] We start with recurrence of~$\mathcal C_\infty$. This property holds a.s. by Theorem~1.4 of~\cite{Berger02}. Then, it holds a.s. on~$\{0 \in \mathcal C_\infty\}$. Hence, applying Theorem~\ref{thm_unimodularity_ICP}, we have that~$\mathcal C_\infty$, on~$\{0 \in \mathcal C_\infty\}$, has the infinite collision property almost surely. We conclude by using~\eqref{transfer_lemma_long_range} with the translation invariant event~$\{\mathcal C_\infty \text{ has the infinite collision property}\}$. 
\end{proof}

\subsection{Regime $d \ge 1$ with $s \in (d,(d+2)\wedge 2d )$}\label{ss_reg_long_tran}
In this regime, we prove that~$\mathcal C_\infty$ has the finite collision property almost surely whenever the edge probabilities are given by a percolating function~$\mathsf P_{\beta,s}$ that satisfies~\eqref{bond_prob_long_range}. To establish this, we use the following result, which is an analogue of Proposition~\ref{thm:bounds_HK_percolation} for the infinite cluster of nearest-neighbor bond percolation. It is a simplified version of Theorem~1 in~\cite{Sly12}, tailored to our setting and needs.
\begin{theorem}\label{thm:heat_kernel_sly}
Under the regime assumptions, there exists an event~$\Omega_1 \subseteq \Omega$ with~$\PP(\Omega_1) = 1$, universal constants~$C_1,\delta > 0$ and a family of random variables~$\{T_x(\omega):x \in \ZZ^d\}$, with the property that~$T_x(\omega) < \infty$ when~$\omega \in \Omega_1$ and~$x \in \mathcal C_\infty(\omega)$, such that for all~$x \in \mathcal C_\infty(\omega)$:
$$ q_t(x,x) \leq C_1 \cdot t^{-d/(s-d)} \log^\delta(t) \quad \text{for } t \ge T_x(\omega). $$
Moreover, for any~$\eta > 0$, there exists~$C(\eta) > 0$ so that for each~$x \in \ZZ^d$, 
$$ \PP( T_x > k \mid  T_x < \infty) \leq C(\eta) \cdot k^{-\eta}. $$
\end{theorem}
\begin{remark}
In the reference, this result is written for edge probabilities satisfying 
$$  p_{x,y} = 1 - \exp(-\beta \norm{x-y}^{-s} ), \quad \norm{x-y} \ge L $$
for some~$L>0$. However, as noted after the statement, the result applies to any edge probabilities that are isotropic (i.e.~$p_{x,y}=\mathsf P(\norm{x-y})$ for some function~$\mathsf P$), translation invariant, have an infinite cluster almost surely, and satisfy:
$$ \lim_{r \to \infty} \frac{ \log \mathsf P(r) }{\log r} = -s. $$
The edge probabilities in our setting satisfy these conditions.
\end{remark}

\begin{proof}[Proof of part two of Theorem~\ref{thm:long_range_cluster}] We begin by establishing the finite collision property. To this end, we apply Theorem~\ref{thm_unimodularity_FCP}, which requires that
\begin{equation}\label{long_range_FCP}
\EE[ \deg(0) \cdot \mathsf G(0,0) \mid 0 \in \mathcal C_\infty ] < \infty.
\end{equation}
To verify this, we use Theorem~\ref{thm:heat_kernel_sly}, which yields
$$  \int^\infty_0 q_t(0,0) \; \mathrm{d}t \leq T_0 + \int^\infty_1 C_1 \cdot t^{-d/(s-d)} \log^\delta(t) \; \mathrm{d}t. $$
The integral on the right-hand side is of Gamma type and is finite under the regime assumptions. Now, using~\eqref{eq_standar_cont_disc} and the definition~\eqref{eq_q_t_heat_kernel_def}, the expected value in~\eqref{long_range_FCP} is smaller than
$$ \left. \EE  \left[ \deg(0)^2 \cdot T_0 \right| 0 \in \mathcal C_\infty \right] + \int^\infty_1 C_1 \cdot t^{-d/(s-d)} \log^\delta(t) \; \mathrm{d}t \cdot \left. \EE  \left[ \deg(0)^2 \right| 0 \in \mathcal C_\infty \right].$$
This expression is finite as a consequence of~\eqref{root_long_range_percolation} and the second part of Theorem~\ref{thm:heat_kernel_sly}. Hence, on~$\{0 \in \mathcal C_\infty\}$,~$\mathcal C_\infty$  has the finite collision property. Since the event~$\{ \mathcal C_\infty \text{ has the finite collision property} \}$ is translation invariant, by~\eqref{transfer_lemma_long_range}, we have that~$\mathcal C_\infty$ has the finite collision property almost surely.
\end{proof}

\subsection*{Acknowledgments}
We would like to thank Tom Hutchcroft for helpful discussions and for pointing out that the unimodular argument does not appear in the literature, although it was known to him. We are also grateful to Perla Sousi for valuable discussions and comments on the results. We further thank Gilles Bonnet and Matthias Irlbeck for references and discussions concerning the Delaunay graph, Alexis Prévost and Daniel Valesin for comments on an earlier version of the manuscript, and Carlos Scali for careful reading of the manuscript and for bringing the reference~\cite{CanCroydonKumagai2022} to our attention. 

J.A. has been supported by FAPESP through grants 2023/13453-5 and 2025/02707-1; the author is thankful for this support. C.M.-A. has been supported by the Deutsche Forschungsgemeinschaft (DFG, German Research Foundation) – Projektnummer 552316285.

\bibliographystyle{abbrv}
    \bibliography{ref}
\end{document}